\documentclass[a4paper]{article}
\usepackage{graphicx}
\usepackage[utf8]{inputenc}
\usepackage[T2A]{fontenc}
\usepackage[english]{babel}
\usepackage{a4wide}
\usepackage{amsmath}
\usepackage{amssymb}
\usepackage{hyperref}
\usepackage{amsthm}
\usepackage{tikz}
\usetikzlibrary{matrix,arrows}
\usepackage{verbatim}
\usepackage{multirow,booktabs}
\usepackage{stmaryrd}
\usepackage{mathrsfs}
\usepackage{cite}

\allowdisplaybreaks[1]

\usetikzlibrary{cd}

\newtheorem{thm}[equation]{Theorem}
\newtheorem{cor}[equation]{Corollary}
\newtheorem{lem}[equation]{Lemma}
\newtheorem{prop}[equation]{Proposition}

\newtheorem{question}[equation]{Question}
\theoremstyle{definition}
\newtheorem{defn}[equation]{Definition}
\newtheorem{exa}[equation]{Example}
\newtheorem{rem}[equation]{Remark}
\newtheorem{nota}[equation]{Notation}

    \newtheoremstyle{TheoremNum}
        {\topsep}{\topsep}              
        {\itshape}                      
        {}                              
        {\bfseries}                     
        {.}                             
        { }                             
        {\thmname{#1}\thmnote{ \bfseries #3}}
    \theoremstyle{TheoremNum}
    \newtheorem{thmn}{Theorem}

\numberwithin{equation}{section}

\newcommand{\recht}[1]{\textnormal{#1}}

\newcommand{\mfr}{\mathfrak}
\newcommand{\msc}{\mathscr}

\newcommand\cat\textsf
\newcommand{\DistTo}{\xrightarrow{\smash{\raisebox{-0.65ex}{\ensuremath{\scriptstyle\sim}}}}}
\DeclareRobustCommand{\VAN}[3]{#2}

\begin{document}

\title{Integral models of reductive groups and integral Mumford--Tate groups}
\author{Milan Lopuhaä-Zwakenberg}
\date{\today}

\maketitle

\section{Introduction}

Let $K$ be a number field or a $p$-adic field, and let $R$ be its ring of integers. 
Let $G$ be a connected reductive group over $K$. 
By a \emph{model} $\msc{G}$ of $G$ we mean a flat group scheme of finite type over $R$ such that $\msc{G}_K \cong G$. 
An important way to construct models of $G$ is the following. 
Let $V$ be a finite dimensional $K$-vector space, and let $\varrho\colon G \hookrightarrow \mathtt{GL}(V)$ be an injective map of algebraic groups (throughout this paper we write $\recht{GL}(V)$ for the abstract group, and $\mathtt{GL}(V)$ for the associated algebraic group over $K$). 
We consider $G$ as an algebraic subgroup of $\mathtt{GL}(V)$ via $\varrho$. Now let $\Lambda$ be a lattice in $V$, i.e. a locally free $R$-submodule of $V$ that generates $V$ as a $K$-vector space. 
Then $\mathtt{GL}(\Lambda)$ is a group scheme over $R$ whose generic fibre is canonically isomorphic to $\mathtt{GL}(V)$. 
Let $\mathtt{mdl}_G(\Lambda)$ be the Zariski closure of $G$ in $\mathtt{GL}(\Lambda)$; this is a model of $G$. 
In general, the group scheme $\mathtt{mdl}_{G}(\Lambda)$ depends on the choice of $\Lambda$, and one can ask the following question:

\begin{question}
Suppose that $G$, its representation $V$, and its model $\mathtt{mdl}_{G}(\Lambda)$ (as an isomorphism class of models of $G$) are given. 
To what extent can we recover the lattice $\Lambda \subset V$?
\end{question}

As a partial answer we can say that the group scheme $\mathtt{mdl}_{G}(\Lambda)$ certainly does not determine $\Lambda$ uniquely. 
Let $g \in \recht{GL}(V)$; then the automorphism $\recht{inn}(g)$ of $\mathtt{GL}(V)$ extends to an isomorphism $\mathtt{GL}(\Lambda) \DistTo \mathtt{GL}(g\Lambda)$. 
As such, we obtain an isomorphism between the group schemes $\mathtt{mdl}_G(\Lambda)$ and $\mathtt{mdl}_{gGg^{-1}}(g\Lambda)$. 
This shows that the group scheme $\mathtt{mdl}_G(\Lambda)$ only depends on the $N(K)$-orbit of $\Lambda$, where $N$ is the scheme-theoretic normaliser of $G$ in $\mathtt{GL}(V)$. 
The following theorem, which is the first main result of this paper, shows that $\mathtt{mdl}_G(\Lambda)$ determines the $N(K)$-orbit of $\Lambda$ up to finite ambiguity.

\begin{thmn}[\ref{thm:main1}]
Let $G$ be a connected reductive group over a number field or $p$-adic field $K$, and let $V$ be a faithful finite dimensional representation of $G$. 
Let $N$ be the scheme-theoretic normaliser of $G$ in $\mathtt{GL}(V)$. 
Let $\msc{G}$ be a model of $G$. 
Then the lattices $\Lambda$ in $V$ such that $\mathtt{mdl}_G(\Lambda) \cong \msc{G}$ are contained in at most finitely many $N(K)$-orbits.
\end{thmn}

In general, a model of $G$ will correspond to more than one $N(K)$-orbit of lattices, see Examples \ref{ex:cl} and \ref{ex:sl}. 

We can apply Theorem \ref{thm:main1} in the context of Shimura varieties and abelian varieties.
Let $g$ and $n > 2$ be positive integers, and let $\mathcal{A}_{g,n}$ be the moduli space of complex principally polarised abelian varieties of dimension $g$ with a given level $n$ structure. 
Let $Y$ be a special subvariety of $\mathcal{A}_{g,n}$, and let $\msc{G} = \mathtt{GMT}(Y)$ be the generic (integral) Mumford--Tate group of $Y$ (with respect to the variation of rational Hodge structures coming from the homology of the universal abelian variety with $\mathbb{Z}$-coefficients). This is an integral model of its generic fibre $\msc{G}_{\mathbb{Q}}$, which is a reductive algebraic group. While reductive groups over fields are well understood, generic integral Mumford--Tate groups are more complicated: there is no general classification of the models of a given rational reductive group, not even for tori (see \cite{fomina1997}). On the other hand, the advantage of the integral group scheme $\mathtt{GMT}(Y)$ is that it carries more information about $Y$ than its generic fibre. This can be seen in \cite[Thm.~4.1]{edixhovenyafaev2003}, where a lower bound is given on the size of the Galois orbit of a CM-point $x$ of a Shimura variety in terms of the reduction of $\mathtt{GMT}(x)$ at finite primes. In this paper, we present another result that highlights the importance of studying the integral group $\mathtt{GMT}(Y)$, by showing that $\mathtt{GMT}(Y)$ determines $Y$ up to a finite ambiguity:

\begin{thmn}[\ref{thm:main2}]
Let $g$ and $n$ be positive integers with $n > 2$, and let $\mathscr{G}$ be a group scheme over~$\mathbb{Z}$. 
Then there are only finitely many special subvarieties $Y$ of $\mathcal{A}_{g,n}$ such that $\mathtt{GMT}(Y) \cong \mathscr{G}$.
\end{thmn}

Note that the integral information is of finite importance here, as the rational group $\mathtt{GMT}(Y)_{\mathbb{Q}}$ 
is invariant under Hecke correspondences. The relation between Theorems \ref{thm:main1} and \ref{thm:main2} can be sketched as follows: Let $G := \mathtt{GMT}(Y)_{\mathbb{Q}}$. The inclusion $Y \hookrightarrow \mathcal{A}_{g,n}$ is induced by a morphism of Shimura data $\varrho\colon(G,X) \hookrightarrow (\mathtt{GSp}_{2g,\mathbb{Q}},\mathcal{H}_g)$ that is injective on the level of algebraic groups. 
The integral variation of Hodge structures defining $\mathtt{GMT}(Y)$ corresponds to a lattice $\Lambda$ in the standard representation $V$ of $\mathtt{GSp}_{2g,\mathbb{Q}}$.
Then $\mathtt{GMT}(Y)$ is isomorphic to $\mathtt{mdl}_G(\Lambda)$, where $V$ is regarded as a faithful representation of $G$ via $\varrho$. 
Replacing $Y$ by a Hecke translate corresponds to replacing the inclusion $G \hookrightarrow \mathtt{GSp}_{2g,\mathbb{Q}}$ by a conjugate, or equivalently, to choosing another lattice in $V$. 
This allows us to apply Theorem \ref{thm:main1} and prove Theorem \ref{thm:main2}, although we first need to make sure that up to a finite choice, Hecke correspondences are the only way to obtain special subvarieties with the same $G$ (see Proposition \ref{prop:main2rat}).

This paper is dedicated to the proofs of these two Theorems. 
In section \ref{s:background} we discuss some facts about models of reductive groups and their relation to Lie algebras. 
In section \ref{s:lie} we briefly recap the representation theory of split reductive groups. 
In section \ref{s:split} we prove Theorem \ref{thm:main1} for split reductive groups over local fields. 
In section \ref{s:local} we extend this result to non-split reductive groups over local fields using Bruhat--Tits theory, and in section \ref{s:numberfield} we prove Theorem \ref{thm:main1} in full generality. 
In section \ref{s:shimura} we prove Theorem \ref{thm:main2}.

\textbf{Acknowledgements:} The research of which this paper is a result was carried out as part of a Ph.D. project at Radboud University supervised by Ben Moonen, to whom I am grateful for comments and guidance. I also thank Johan Commelin and Maarten Solleveld for further comments. All remaining errors are, of course, my own.

\section{Lattices, models, and Lie algebras} \label{s:background}

In this section we will discuss a number of properties of models, and their relation to lattices in various vector spaces. 
Throughout we fix a number field or $p$-adic field $K$, along with its ring of integers $R$.

\subsection{Models of reductive groups}

\begin{defn} Let $G$ be a connected reductive algebraic group over $K$, and let $T$ be a maximal torus of $G$.
\begin{enumerate}
\item A \emph{model} of $G$ is a flat group scheme $\mathscr{G}$ of finite type over $R$ such that there exists an isomorphism $\varphi\colon \mathscr{G}_K \DistTo G$. 
Such an isomorphism is called an \emph{anchoring} of $\mathscr{G}$. 
The set of isomorphism classes of models of $G$ is denoted $\recht{Mdl}(G)$.
\item An \emph{anchored model} of $G$ is pair $(\mathscr{G},\varphi)$ of a model $\mathscr{G}$ and an anchoring $\varphi$. 
The set of isomorphism classes of anchored models of $G$ is denoted $\recht{Mdl}^{\recht{a}}(G)$.
\item A \emph{model} of $(G,T)$ is a pair $(\mathscr{G},\mathscr{T})$ of a model of $G$ and a closed reduced subgroup scheme $\mathscr{T}$ of $\mathscr{G}$, for which there exists an isomorphism $\varphi\colon \mathscr{G}_K \DistTo G$ such that $\varphi|_{\mathscr{T}_K}$ is an isomorphism $\mathscr{T}_K \DistTo T$. 
Such a $\varphi$ is called an \emph{anchoring} of $(\mathscr{G},\mathscr{T})$. 
The set of isomorphism classes of models of $(G,T)$ is denoted $\recht{Mdl}(G,T)$.
\end{enumerate}
\end{defn}

Note that there are natural forgetful maps $\recht{Mdl}^{\recht{a}}(G) \rightarrow \recht{Mdl}(G,T) \rightarrow \recht{Mdl}(G)$. 
Our use of the terminology `model' may differ from its use in the literature; 
for instance, some authors consider the choice of an anchoring to be part of the data (hence their `models' would be our `anchored models'), 
or they may impose other conditions on the group scheme $\mathscr{G}$ over $R$, see for instance \cite{conrad2011,fomina1997,gross1996}.
Our choice of terminology is justified by the fact that our models are exactly those that arise from lattices in representations (see Remark \ref{rem:sga}).

\begin{defn}
Let $V$ be a $K$-vector space. 
A \emph{lattice} in $V$ is a locally free $R$-submodule of $V$ that spans $V$ as a $K$-vector space. 
The set of lattices in $V$ is denoted $\recht{Lat}(V)$. 
If $H \subset \mathtt{GL}(V)$ is an algebraic subgroup, we write $\recht{Lat}_H(V)$ for the quotient $H(K) \backslash \recht{Lat}(V)$. 
\end{defn}


Let $G$ be a connected reductive group over $K$, and let $V$ be a finite dimensional faithful representation of $G$; 
we consider $G$ as an algebraic subgroup of $\mathtt{GL}(V)$. 
Let $\Lambda$ be a lattice in $V$. 
The identification $\Lambda_K = V$ induces a natural isomorphism $f_{\Lambda}\colon \mathtt{GL}(\Lambda)_K \DistTo \mathtt{GL}(V)$.
 Let $\mathtt{mdl}_G(\Lambda)$ be the Zariski closure of $f_{\Lambda}^{-1}(G)$ in $\mathtt{GL}(\Lambda)$; this is a model of $G$. If we let $\varphi_{\Lambda}$ be the isomorphism $f_{\Lambda}|_{\mathtt{mdl}_G(\Lambda)_K}\colon \mathtt{mdl}_G(\Lambda)_K \DistTo G$, then $(\mathtt{mdl}_G(\Lambda),\varphi_{\Lambda})$ is an anchored model of $G$. This gives us a map
\begin{align*}
\mathtt{mdl}^{\recht{a}}_G\colon \recht{Lat}(V) &\rightarrow \recht{Mdl}^{\recht{a}}(G) \\
\Lambda &\mapsto (\mathtt{mdl}_G(\Lambda),\varphi_{\Lambda}).
\end{align*}
The compositions of $\mathtt{mdl}^{\recht{a}}_G$ with the forgetful maps $\recht{Mdl}^{\recht{a}}(G) \rightarrow \recht{Mdl}(G,T)$ (for a maximal torus $T$ of $G$) and $\recht{Mdl}^{\recht{a}}(G) \rightarrow \recht{Mdl}(G)$ are denoted $\mathtt{mdl}_{G,T}$ and $\mathtt{mdl}_G$, respectively.

\begin{lem} \label{lem:mdldef}
Let $G$ be a connected reductive group over $K$ and let $V$ be a faithful finite dimensional representation of $G$. 
Consider $G$ as a subgroup of $\mathtt{GL}(V)$. 
Let $Z$ be the scheme-theoretic centraliser of $G$ in $\mathtt{GL}(V)$, and let $N$ be the scheme-theoretic normaliser of $G$ in $\mathtt{GL}(V)$. 
Let $T$ be a maximal torus of $G$, and let $H := Z \cdot T \subset \mathtt{GL}(V)$.
\begin{enumerate}
\item The map $\mathtt{mdl}_G^{\recht{a}}\colon \recht{Lat}(V) \rightarrow \recht{Mdl}^{\recht{a}}(G)$ factors through $\recht{Lat}_Z(V)$.
\item The map $\mathtt{mdl}_{G,T}\colon \recht{Lat}(V) \rightarrow \recht{Mdl}(G,T)$ factors through $\recht{Lat}_H(V)$.
\item The map $\mathtt{mdl}_G\colon \recht{Lat}(V) \rightarrow \recht{Mdl}(G)$ factors through $\recht{Lat}_N(V)$.
\end{enumerate}
\end{lem}

\begin{proof}
We only prove the first statement; the other two can be proven analogously. 
Let $g $ be an element of $\recht{GL}(V)$. 
The map $\recht{inn}(g) \in \recht{Aut}(\mathtt{GL}(V))$ extends to an automorphism $\mathtt{GL}(\Lambda) \rightarrow \mathtt{GL}(g\Lambda)$ as in the following diagram:

\begin{center}
\begin{tikzpicture}[description/.style={fill=white,inner sep=2pt}, bij/.style={below,sloped,inner sep=2pt}]
\matrix (m) [matrix of math nodes, row sep=3em,
column sep=6em, text height=1.5ex, text depth=0.25ex]
{ \mathtt{GL}(\Lambda) & \mathtt{GL}(g\Lambda) \\
\mathtt{GL}(\Lambda)_K & \mathtt{GL}(g\Lambda)_K \\
\mathtt{GL}(V) & \mathtt{GL}(V) \\
G & gGg^{-1} \\};
\path[->,font=\scriptsize]
(m-1-1) edge node[auto]{$(f_{g\Lambda} \circ \recht{inn}(g) \circ f_{\Lambda}^{-1})^{\recht{zar}}$} node[bij]{$\sim$} (m-1-2)
(m-2-1) edge (m-1-1)
(m-2-1) edge node[auto]{$f_{g\Lambda} \circ \recht{inn}(g) \circ f_{\Lambda}^{-1}$} node[bij]{$\sim$} (m-2-2)
edge node[auto] {$f_{\Lambda}$} node[bij]{$\sim$} (m-3-1)
(m-2-2) edge node[auto]{$f_{g\Lambda}$} node[bij]{$\sim$} (m-3-2)
(m-3-1) edge node[auto]{$\recht{inn}(g)$} node[bij]{$\sim$} (m-3-2) 
(m-4-1) edge node[bij]{$\sim$} (m-4-2);
\path[right hook->]
(m-2-1) edge (m-1-1)
(m-2-2) edge (m-1-2)
(m-4-1) edge (m-3-1)
(m-4-2) edge (m-3-2);
\end{tikzpicture}
\end{center}

This shows that $(\mathtt{mdl}_G(\Lambda),\varphi_{\Lambda}) \cong (\mathtt{mdl}_{gGg^{-1}}(g\Lambda),\recht{inn}(g)^{-1} \circ f_{g\Lambda})$ as anchored models of $G$. 
If $g$ is an element of $Z(K)$ we find that $(\mathtt{mdl}_G(\Lambda),\varphi_{\Lambda}) \cong (\mathtt{mdl}_G(g\Lambda),\varphi_{g\Lambda})$, as was to be proven.
\end{proof}

Throughout the rest of this paper we say that a map of sets if \emph{finite} if it has finite fibres. 
Our first main result, whose proof will be finalized in Section \ref{s:numberfield}, states that up to a finite ambiguity, the model $\mathtt{mdl}_G(\Lambda)$ characterises the $N(K)$-orbit of $\Lambda$.

\begin{thm} \label{thm:main1}
Let $G$ be a connected reductive group over a number field or $p$-adic field $K$, and let $V$ be a faithful finite dimensional representation of $G$. Let $N$ be the scheme-theoretic normaliser of $G$ in $\mathtt{GL}(V)$. Then the map $\mathtt{mdl}_G\colon \recht{Lat}_N(V) \rightarrow \recht{Mdl}(G)$ is finite.
\end{thm}

\begin{exa} \label{ex:cl}
Let $F$ be a number field, and let $G = \recht{Res}_{F/\mathbb{Q}}(\mathbb{G}_{\recht{m}})$ be the Weil restriction of $\mathbb{G}_{\recht{m}}$ from $F$ to $\mathbb{Q}$. 
Let $V$ be the $\mathbb{Q}$-vector space $F$, together with its natural representation of $G$. 
Let $\Lambda$ be a lattice in $V$, and define the ring $A_{\Lambda} := \{x \in F\colon x\Lambda \subset \Lambda\}$; this is an order in $F$. 
In this case one has $\mathtt{mdl}_G(\Lambda) \cong \recht{Res}_{A_{\Lambda}/\mathbb{Z}}(\mathbb{G}_{\recht{m}})$ as group schemes over $\mathbb{Z}$. 
Now let $\Lambda$ be such that $A_{\Lambda} = \mathscr{O}_F$. 
As a subgroup of $F$, the lattice $\Lambda$ can be considered as a fractional $\mathscr{O}_F$-ideal. 
Since in this case we have $N(\mathbb{Q}) = G(\mathbb{Q}) = F^\times$, the $N(\mathbb{Q})$-orbit of $\Lambda$ corresponds to an element of the class group $\recht{Cl}(F)$. 
On the other hand, every element of $\recht{Cl}(F)$ corresponds to a $N(\mathbb{Q})$-orbit of lattices $\Lambda$ in $V$ satisfying $A_{\Lambda} = \mathscr{O}_F$. 
In other words, there is a bijective correspondence between $N(\mathbb{Q})$-orbits of lattices yielding the model $\recht{Res}_{\mathscr{O}_{F}/\mathbb{Z}}(\mathbb{G}_{\recht{m}})$ of $G$, and elements of the class group $\recht{Cl}(F)$. 
This shows that a model of $G$ generally does not correspond to a single $N$-orbit of lattices. In this setting, Theorem \ref{thm:main1} recovers the well-known fact that $\recht{Cl}(F)$ is finite.
\end{exa}

\begin{rem} \label{rem:sga}
Let $G$ be a (not necessarily connected) reductive group over $K$, and let $\mathscr{G}$ be a model of $G$. 
Then \cite[Exp.~VI.B, Prop.~13.2]{sga3} tells us that there exists a free $R$-module $\Lambda$ of finite rank such that $\mathscr{G}$ is isomorphic to a closed subgroup of $\mathtt{GL}(\Lambda)$. 
If we take $V =\Lambda_K$, we find that $V$ is a faithful representation of $G$, and that $\mathscr{G}$ is the image of $\Lambda$ under the map $\mathtt{mdl}_G \colon \recht{Lat}(V) \rightarrow \recht{Mdl}(G)$. 
Hence every model of $G$ arises from a lattice in some representation.
\end{rem}

An important class of models that will turn up frequently in our discussions is given in the Definition below.

\begin{defn} \label{def:chevalley}
Suppose $G$ is a split reductive group with a split maximal torus $T$.
In that case there is exactly one model $(\mathscr{G},\mathscr{T})$ of $(G,T)$ such that $\mathscr{G}$ is reductive (i.e. smooth with reductive fibres) and such that $T$ is a split fibrewise maximal torus \cite[Exp. XXIII, Cor. 5.2; Exp. XXV, Cor. 1.2]{sga3}. 
This model is called the \emph{Chevalley model} of $(G,T)$. 
We also refer to $\mathscr{G}$ as the Chevalley model of $G$.
\end{defn}

\begin{exa} \label{ex:sl}
We give an example that shows that the map $\mathtt{mdl}\colon \recht{Lat}_N(V) \rightarrow \recht{Mdl}(G)$ is generally not injective even when only considering $p$-adic fields. Take $K = \mathbb{Q}_2$, and let $G$ be the algebraic group $\mathtt{PGL}_{2,\mathbb{Q}_2}$. The standard representation $V$ of $\tilde{G} = \mathtt{SL}_{2,\mathbb{Q}_2}$ induces a representation of $G$ on $W = \recht{Sym}^2(V)$. Let $\tilde{A} := \mathcal{O}(\tilde{G}) = \mathbb{Q}_2[x_{11},x_{12},x_{21},x_{22}]/(x_{11}x_{22}-x_{12}x_{21}-1)$, then $A := \mathcal{O}(G)$ is its $K$-subalgebra generated by monomials of even degree. If $\Lambda$ is a lattice in $W$, then $\recht{mdl}_G(\Lambda) \cong \recht{Spec}(\msc{A})$, where $\msc{A}$ is the image of the composite map $\mathcal{O}(\mathtt{GL}(\Lambda)) \hookrightarrow \mathcal{O}(\mathtt{GL}(V)) \twoheadrightarrow A$. Let $\{e_1,e_2\}$ be the standard basis of $V$. Then one can show that the lattices $\Lambda$ and $\Lambda'$ in $W$, generated by $\{e_1^2,e_1e_2,e_2\}$ and $\{e_1^2,2e_1e_2,e_2\}$ respectively, both induce the same $\msc{A}$, namely the $R$-subalgebra of $A$ generated by
\begin{equation*}
\Big\{x_{11}^2,x_{11}x_{12},x_{12}^2,x_{11}x_{21},x_{11}x_{22}+x_{12}x_{21},x_{12}x_{22},x_{21}^2,x_{21}x_{22},x_{22}^2\Big\}.
\end{equation*}
On the other hand, the normaliser of $G$ is equal to $N = T \cdot \sigma(\mathtt{GL}_{2,\mathbb{Q}_2})$, where $T \subset \mathtt{GL}(W)$ is the group of scalars and $\sigma\colon\mathtt{GL}_{2,\mathbb{Q}_2} \rightarrow \mathtt{GL}(W)$ is the natural representation. One can show that each element of $N(\mathbb{Q}_2)$ can be written (non-uniquely) as $t \cdot \sigma(g)$, where $t \in \mathbb{Q}_2^{\times}$ and $g \in \recht{GL}_2(\mathbb{Q}_2)$. For such an element, one has $\recht{det}(t \cdot \sigma(g)) = t^3\recht{det}(g)^3$. However, one has $\#(\Lambda/\Lambda') = 2$, and since $2$ is not a cube in $\mathbb{Q}_2$ this means that $\Lambda$ and $\Lambda'$ cannot be in the same $N(\mathbb{Q}_2)$-orbit. 
\end{exa}

\subsection{Lie algebras}

Let $\mathfrak{g}$ be the ($K$-valued points of the) Lie algebra of $G$. 
Let $\mathscr{G}$ be a model of $G$, and let $\mfr{G}$ be the ($R$-valued points of the) Lie algebra of $\msc{G}$. 
Then $\mathfrak{g}$ is a $K$-vector space of dimension $\recht{dim}(G)$, and $\mfr{G}$ is a locally free $R$-module of rank $\recht{dim}(G)$. 
If $\varphi$ is an anchoring of $\msc{G}$, then $\varphi$ induces an embedding of $R$-Lie algebras $\recht{Lie}\ \varphi\colon \mfr{G} \hookrightarrow \mfr{g}$, and its image is a lattice in $\mfr{g}$. 
Suppose $V$ is a faithful representation of $G$ and $\Lambda \subset V$ is a lattice such that $\mathtt{mdl}_{\recht{a}}(\Lambda) = (\msc{G},\varphi)$.
Then $(\recht{Lie} \ \varphi)(\mfr{G}) = \mfr{g} \cap \mfr{gl}(\Lambda)$ as subsets of $\mfr{gl}(V)$.

\subsection{Lattices in vector spaces over $p$-adic fields} \label{ss:dist}

Suppose $K$ is a $p$-adic field, and let $\pi$ be a uniformiser of $K$. 
Let $V$ like before be a finite dimensional $K$-vector space, and let $\Lambda$, $\Lambda'$ be two lattices in $V$. 
Then there exist integers $n,m$ such that $\pi^n\Lambda \subset \Lambda' \subset \pi^m \Lambda$. 
If we choose $n$ minimal and $m$ maximal, then we call $d(\Lambda,\Lambda') := n-m$ the \emph{distance} between $\Lambda$ and $\Lambda'$. 
Let $G$ be an algebraic subgroup of $\mathtt{GL}(V)$, and as before let $\recht{Lat}_G(V) = G(K) \backslash \recht{Lat}(V)$. 
We define a function
\begin{align*}
d_G\colon \recht{Lat}_G(V) \times \recht{Lat}_G(V) &\rightarrow \mathbb{R}_{\geq 0} \\
(X,Y) &\mapsto \min_{(\Lambda,\Lambda') \in X \times Y} d(\Lambda,\Lambda').
\end{align*}

The following lemma tells us that the name `distance' is justified. 
Its proof is straightforward and therefore omitted.

\begin{lem} \label{lem:dist} 
Let $V$ and $G$ be as above. Suppose $G$ contains the scalars in $\mathtt{GL}(V)$.
\begin{enumerate}
\item Let $X,Y \in \recht{Lat}_G(V)$ and let $\Lambda \in X$. Then $d_G(X,Y) = \min_{\Lambda' \in Y}d(\Lambda,\Lambda')$.
\item The map $d_G$ is a distance function on $\recht{Lat}_G(V)$.
\item For every $r \in \mathbb{R}_{\geq 0}$ and every $Y \in \recht{Lat}_G(V)$ the open ball $\{X \in \recht{Lat}_G(V):d_G(X,Y) < r\}$ is finite. \qed
\end{enumerate}
\end{lem}

\section{Representations of split reductive groups} \label{s:lie}

As before let $K$ be a number field or a $p$-adic field. 
In this section we will briefly review the representation theory of split reductive groups over $K$, and we will set up the notation for the (quite technical) next chapter.
Furthermore, we will prove some results on the associated representation theory of Lie algebras. 
We will assume all representations to be finite dimensional.

Let $G$ be a connected split reductive group over $K$, and let $T \subset G$ be a split maximal torus. 
Furthermore, we fix a Borel subgroup $B \subset G$ containing $T$. 
Let $\Psi \subset \recht{X}^*(T)$ be the set of roots of $G$ with respect to $T$; 
let $Q \subset \recht{X}^*(T)$ be the subgroup generated by $\Psi$. 
Associated to $B$ we have a basis $\Delta^+$ of $\Psi$ such that every $\beta \in \Psi$ can be written as $\beta = \sum_{\alpha \in \Delta^+} m_{\alpha} \alpha$, with the $m_{\alpha}$ either all nonpositive integers or all nonnegative integers.
 This gives a decomposition $\Psi = \Psi^+ \sqcup \Psi^-$. 
 Accordingly, if $\mfr{g}$ and $\mfr{t}$ are the Lie algebras of $G$ and $T$, respectively, we get 
 \begin{equation*}
 \mfr{g} = \mfr{t} \oplus \mfr{n}^+ \oplus \mfr{n}^- := \mfr{t} \oplus \left(\bigoplus_{\alpha \in \Psi^+} \mfr{g}_{\alpha}\right)\oplus \left(\bigoplus_{\alpha \in \Psi^-} \mfr{g}_{\alpha}\right).
 \end{equation*}
 The following theorem gives a description of the irreducible representations of $G$. 
If $V$ is a representation of $G$, we call the characters of $T$ that occur in $V$ the \emph{weights} of $V$ (with respect to $T$).

\begin{thm} \label{thm:irred}
Let $V$ be an irreducible representation of $G$.
\begin{enumerate}
\item \emph{(See \cite[Thm. 24.3]{milne2017})} There is a unique weight $\psi$ of $V$, called the \emph{highest weight} of $V$, such that $V_{\psi}$ is one-dimensional, and every weight of $V$ is of the form $\psi - \sum_{\alpha \in \Delta^+} m_{\alpha} \alpha$ for constants $m_{\alpha} \in \mathbb{Z}_{\geq 0}$.
\item \emph{(See \cite[{}3.39]{milne2013})} $V$ is irreducible as a representation of the Lie algebra $\mfr{g}$.
\item \emph{(See \cite[Ch.~VIII, §6.1, Prop.~1]{bourbaki2006})} $V$ is generated by the elements obtained by repeatedly applying $\mfr{n}^-$ to $V_{\psi}$. 
\item \emph{(See \cite[Thm. 24.3]{milne2017})} Up to isomorphism $V$ is the only irreducible representation of $G$ with highest weight $\psi$.\qed
\end{enumerate}
\end{thm}

\begin{rem} \label{rem:decompsplit}
With $G$ as above, let $V$ be any representation of $G$. 
Then, because $G$ is reductive, we know that $V$ is a direct sum of irreducible representations of $G$. 
By Theorem \ref{thm:irred}.1 we can canonically write $V = \bigoplus_{\psi \in \mathcal{D}} V_{(\psi)}$, where for $\psi \in \recht{X}^*(T)$ the subspace $V_{(\psi)}$ is the isotypical component of $V$ with highest weight $\psi$ (as a character of $T$), and $\mathcal{D}$ is the set of highest weights occuring in $V$. 
Furthermore, we can decompose every $V_{(\psi)}$ into $T$-character spaces, and we get a decomposition $V = \bigoplus_{\psi \in \msc{D}}\bigoplus_{\chi \in \recht{X}^*(T)} V_{(\psi),\chi}$.
\end{rem}

Let $U(\mfr{g})$ be the universal enveloping algebra of $\mfr{g}$. 
It obtains a $Q$-grading coming from the $Q$-grading of $\mfr{g}$; 
we may also regard this as a $\recht{X}^*(T)$-grading via the inclusion $Q \subset \recht{X}^*(T)$. 
If $V$ is a representation of $G$, then the associated map $U(\mfr{g}) \rightarrow \recht{End}(V)$ is a homomorphism of $\recht{X}^*(T)$-graded $K$-algebras. 
Furthermore, from the Poincaré--Birkhoff--Witt theorem it follows that there is a natural isomorphism of $Q$-graded $K$-algebras $U(\mfr{g}) \cong U(\mfr{n}^-) \otimes U(\mfr{t}) \otimes U(\mfr{n}^+)$, with the map from right to left given by multiplication. 
The following two results will be useful in the next section.

\begin{thm}[Jacobson density theorem] \label{thm:jacobson} Let $G$ be a split reductive group, and let $\mfr{g}$ be its Lie algebra. 
Let $V_1,\ldots,V_n$ be pairwise nonisomorphic irreducible representations of $G$. Then the induced map $U(\mfr{g}) \rightarrow \bigoplus_i \recht{End}(V_i)$ is surjective.
\end{thm}

\begin{proof}
This theorem is proven over algebraically closed fields in \cite[Thm.~2.5]{etingof2011} for representations of algebras in general (not just for universal enveloping algebras of Lie algebras). 
The hypothesis that $K$ is algebraically closed is only used in invoking Schur's Lemma, but this also holds for split representations of reductive groups.
\end{proof}

\begin{prop} \label{prop:surj}
Let $V$ be an irreducible representation of $G$ of highest weight $\psi$. 
Let $\chi$ be a weight of $V$. 
Then the maps 
$U(\mfr{n}^-)_{\chi-\psi} \rightarrow \recht{Hom}_K(V_{\psi},V_{\chi})$ and $U(\mfr{n}^+)_{\psi-\chi} \rightarrow \recht{Hom}_K(V_\chi,V_\psi)$ are surjective.
\end{prop}

\begin{proof}
From Theorem \ref{thm:irred}.3 we know that $V = U(\mfr{n}^-) \cdot V_{\psi}$. 
Since $U(\mfr{n}^-) \rightarrow \recht{End}(V)$ is a homomorphism of $\recht{X}^*(T)$-graded $K$-algebras, this implies that $V_{\chi} = U(\mfr{n}^-)_{\chi-\psi} \cdot V_{\psi}$. 
Since $V_{\psi}$ is one-dimensional by Theorem \ref{thm:irred}.1 this shows that $U(\mfr{n}^-)_{\chi-\psi} \rightarrow \recht{Hom}_K(V_{\psi},V_{\chi})$ is surjective.

For the surjectivity of the second map, let $f\colon V_{\chi} \rightarrow V_{\psi}$ be a linear map, and extend $f$ to a map $\tilde{f}\colon V \rightarrow V$ by letting $\tilde{f}$ be trivial on all $V_{\chi'}$ with $\chi' \neq \chi$. 
Then $\tilde{f}$ is pure of degree $\psi-\chi$, and $\psi-\chi \in Q$ by Theorem \ref{thm:irred}.1. 
By Theorem \ref{thm:jacobson} there exists a $u \in U(\mfr{g})_{\psi-\chi}$ such that the image of $u$ in $\recht{End}(V)$ equals $\tilde{f}$. 
We know that $U(\mfr{g}) = U(\mfr{n}^-) \otimes U(\mfr{t}) \otimes U(\mfr{n}^+)$; 
write $u = \sum_{i \in I} u_i^- \cdot t_i \cdot u_i^+$ with each $u_i^-$, $t_i$ and $u_i^+$ of pure degree, such that each $u_i^- \cdot t_i \cdot u_i^+$ is of degree $\psi-\chi$. 
Let $I'$ be the subset of $I$ of the $i$ for which $u_i^+$ is of degree $\psi-\chi$. 
Since only negative degrees (i.e. sums of nonpositive multiples of elements of $\Delta^+$) occur in $U(\mfr{n}^-)$ and only degree 0 occurs in $U(\mfr{t})$, this means that $u_i^-$ is of degree 0 for $i \in I'$; 
hence for these $i$ the element $u_i^-$ is a scalar. 
Now consider the action of $u$ on $V_\chi$. 
If $i \notin I'$, then the degree of $u_i^+$ will be greater than $\psi - \chi$, in which case we will have $u_i^+ \cdot V_{\chi} = 0$. 
For all $v \in V_\chi$ we now have
\begin{align*}
\tilde{f}(v) = u \cdot v &= \left(\sum_{i \in I} u_i^- \cdot t_i \cdot u_i^+\right) \cdot v \\
&= \left(\sum_{i \in I'} u_i^- \cdot t_i \cdot u_i^+\right) \cdot v \\
&= \left(\sum_{i \in I'} u_i^-\psi(t_i)u_i^+\right) \cdot v.
\end{align*}
Because every factor $u_i^-$ in this sum is a scalar, we know that $\sum_{i \in I'} u_i^-\psi(t_i)u_i^+$ is an element of $U(\mfr{n}^+)_{\chi-\psi}$, and it acts on $V_{\chi}$ as the map $f \in \recht{Hom}_K(V_\chi,V_\psi)$; 
hence the map $U(\mfr{n}^+)_{\psi-\chi} \rightarrow \recht{Hom}_K(V_\chi,V_\psi)$ is surjective.
\end{proof}

\section{Split reductive groups over local fields} \label{s:split}

In the rest of this chapter $K$ is either a number field or a $p$-adic field, and $R$ is its ring of integers. 
All representations of algebraic groups are assumed to be finite dimensional. 
The aim of this section is to prove the following Theorem.

\begin{thm} \label{thm:split}
Let $G$ be a \emph{split} connected reductive group over $K$, and let $V$ be a faithful representation of $G$. 
Regard $G$ as a subgroup of $\mathtt{GL}(V)$, and let $N$ be the scheme-theoretic normaliser of $G$ in $\mathtt{GL}(V)$.
\begin{enumerate}
\item Suppose $K$ is a $p$-adic field. 
Then the map $\mathtt{mdl}_G\colon \recht{Lat}_N(V) \rightarrow \recht{Mdl}(G)$ of Lemma \ref{lem:mdldef} is finite.
\item Suppose $K$ is a number field. 
Then for all but finitely many finite places $v$ of $K$ there is at most one $N(K_v)$-orbit $X$ of lattices in $V_{K_v}$ such that $\mathtt{mdl}_G(X)$ is the Chevalley model of $G$ (see Definition \ref{def:chevalley}).
\end{enumerate}
\end{thm}

The first point is Theorem \ref{thm:main1} for split reductive groups over local fields. 
The second point is technical by itself, but we need this finiteness result to combine the local information in order to Theorem \ref{thm:main1} for number fields. 
The proof is quite involved, and the overall strategy is as follows. 
Let $\mfr{g}$ be the Lie algebra of $G$, then $V$ is a faithful representation of $\mfr{g}$. 
If $\msc{G}$ is any model of $G$, then its Lie algebra $\mfr{G}$ is a lattice in $\mfr{g}$.
The proofs consists of the following steps:
\begin{enumerate}
\item We define a set of `nice' lattices in $\mfr{g}$ (Definition \ref{def:chevlat}) and a set of `nice' lattices in $V$ (Definition \ref{def:chevinvlat});
\item We show that the `nice' lattices in $V$ form only finitely many orbits under a suitably chosen algebraic group (Proposition \ref{prop:finhorbloc});
\item We show that if $\msc{G}$ is the model corresponding to a lattice $\Lambda$, then we can give an upper bound to the distance between $\Lambda$ and a `nice' lattice in $V$, in terms of the distance between $\msc{G}$ and a `nice' lattice in $\mfr{g}$ (Proposition \ref{prop:torusboundloc}).
\end{enumerate}
This upper bound allows us to prove that there are only finitely many $N(K)$-orbits corresponding to one model $\msc{G}$.

\subsection{Lattices in representations} \label{ss:lat}

In this section we will introduce two important classes of lattices that occur in representations of split reductive groups. 
We will rely on much of the results and notations from section \ref{s:lie}.

\begin{nota} \label{not:rep}
For the rest of section \ref{s:split}, we fix the following objects and notation:
\begin{itemize}
 \itemsep0em 
 \item a split connected reductive group $G$ over $K$ and a split maximal torus $T \subset G$;
 \item the Lie algebras $\mfr{g}$ and $\mfr{t}$ of $G$ and $T$, respectively;
 \item the root system $\Psi \subset \recht{X}^*(T)$ of $G$ with respect to $T$, and the subgroup $Q$ of $\recht{X}^*(T)$ generated by $\Psi$;
 \item the image $\bar{T}$ of $T$ in $G^\recht{ad} \subset \mathtt{GL}(\mfr{g})$;
 \item the decomposition $\mfr{g} = \mfr{t} \oplus \bigoplus_{\alpha \in \Psi} \mfr{g}_{\alpha}$;
 \item the basis of positive roots $\Delta^+$ of $\Psi$ associated to some Borel subgroup $B$ of $G$ containing $T$, the decompositions $\Psi = \Psi^+ \sqcup \Psi^-$ and $\mfr{g} = \mfr{t} \oplus \mfr{n}^+ \oplus \mfr{n}^-$;
 \item the $Q$-graded universal enveloping algebra $U(\mfr{g})$ of $\mfr{g}$;
 \item a faithful representation $V$ of $G$ and its associated inclusion $\mfr{g} \subset \mfr{gl}(V)$;
 \item the centraliser $Z$ of $G$ in $\mathtt{GL}(V)$, and the algebraic group $H = Z \cdot T \subset \mathtt{GL}(V)$;
 \item the decomposition $V = \bigoplus_{\psi \in \mathcal{D}} \bigoplus_{\chi \in \recht{X}^*(T)} V_{(\psi),\chi}$ (see Remark \ref{rem:decompsplit});
 \item the projections $\recht{pr}_{(\psi),\chi}\colon V \rightarrow V_{(\psi),\chi}$ associated to the decomposition above.
 \end{itemize}
 \end{nota}
 
\begin{rem} \label{rem:not} 
\begin{enumerate}
\item Since the set of characters of $T$ that occur in the adjoint representation is equal to $\{0\} \cup \Psi$, the inclusion $\recht{X}^*(\bar{T}) \hookrightarrow \recht{X}^*(T)$ has image $Q$.
\item By Schur's lemma the induced map $Z \rightarrow \prod_{\psi \in \msc{D}} \recht{GL}(V_{(\psi),\psi})$ is an isomorphism.
\end{enumerate}
\end{rem}

\begin{defn}
Let $W$ be a $K$-vector space with a decomposition $W = \bigoplus_i W_i$. 
An $R$-submodule $M \subset V$ is called \emph{split} with respect to this decomposition if one of the following equivalent conditions is satisfied:
\begin{enumerate}
 \itemsep0em 
 \item $M = \bigoplus_{i} \recht{pr}_{i} M$;
 \item $M = \bigoplus_{i} (W_i \cap M)$.
 \end{enumerate}
 If $M$ is split, we write $M_i := \recht{pr}_i M = W_i \cap M$.
 \end{defn}
 We now define two classes of lattices that will become important later on. 
 Since the Lie algebra $\mfr{g}$ is a $K$-vector space, we can consider lattices in $\mfr{g}$. 
 For a vector space $W$ over $K$, let $\recht{FLat}(W)$ be the set of lattices in $W$ that are free as $R$-modules. 
 Define the following sets:
 \begin{align*}
\mathcal{L}^+ &:= \prod_{\alpha \in \Delta^+} \recht{FLat}(\mfr{g}_\alpha);\\
\mathcal{L}^- &:= \prod_{\alpha \in \Delta^+} \recht{FLat}(\mfr{g}_{-\alpha});\\
\mathcal{J} &:= \prod_{\psi \in \mathcal{D}} \recht{FLat}(V_{(\psi),\psi}).
\end{align*}
As before, let $U(\mfr{n}^+)$ be the universal enveloping algebra of $\mfr{n}^+$. 
Let $L^+ = (L^+_{\alpha})_{\alpha \in \Delta^+}$ be an element of $\mathcal{L}^+$, 
and let $\msc{U}_{L^+}$ be the $R$-subalgebra of $U(\mfr{n}^+)$ generated by the $R$-submodules $L_{\alpha}^+ \subset \mfr{n}^+$. 
Define, for an $L^- \in \mathcal{L}^-$, the $R$-subalgebra $\msc{U}_{L^-} \subset U(\mfr{n}^-)$ analogously. 
Now let $L^+ \in\mathcal{L}^+$, $L^- \in \mathcal{L}^-$ and $J \in \mathcal{J}$ be as above. 
We define the following two $R$-submodules of $V$:
\begin{align*}
S^+(L^+,J) &:= \Big\{x \in V: \forall \psi \in \mathcal{D}, \recht{pr}_{(\psi),\psi}(\msc{U}_{L^+} \cdot x) \subset J_{\psi}\Big\},\\
S^-(L^-,J) &:= \sum_{\psi \in \msc{D}} \msc{U}_{L^-} \cdot J_{\psi} \subset V.
\end{align*}

Note that the sum in the second equation is actually direct, since $\msc{U}_{L^-} \cdot J_\psi \subset V_{(\psi)}$ for all $\psi \in \mathcal{D}$. 

\begin{prop} \label{prop:spm}
Let $L^+ \in \mathcal{L}^+$, $L^- \in \mathcal{L}^-$ and $J\in \mathcal{J}$.
\begin{enumerate}
\item $\msc{U}_{L^{\pm}}$ is a split lattice in $U(\mfr{n}^{\pm})$ with respect to the $Q$-grading.
\item $S^{\pm}(L^{\pm},J)$ is a split lattice in $V$ with respect to the decomposition $V = \bigoplus_{\psi,\chi} V_{(\psi),\chi}$.
\item For all $\psi \in \mathcal{D}$ one has $S^{\pm}(L^{\pm},J)_{(\psi),\psi} = J_{\psi}$. 
Furthermore, $S^+(L^+,J)$ (respectively $S^-(L^-,J)$) is the maximal (respectively minimal) split lattice $\Lambda$ in $V$ invariant under the action of the $L^+_{\alpha}$ (respectively the $L^-_{\alpha}$) such that $\Lambda_{(\psi),\psi} = J_{\psi}$ for all $\psi \in \mathcal{D}$.
\end{enumerate}
\end{prop}

\begin{proof} 
\begin{enumerate}
\item It suffices to prove this for $\msc{U}_{L^+}$. 
Recall that $U(\mfr{n}^+)$ has an $Q$-grading coming from the $Q$-grading on $U(\mfr{g})$. 
Since $U(L^+)$ is generated by elements of pure degree, we see that $\msc{U}_{L^+}$ is split with respect to the $Q$-grading; 
hence it suffices to show that $\msc{U}_{L^+,\chi}$ is a lattice in  $U(\mfr{n}^+)_{\chi}$ for all $\chi$. 
Since each $\mfr{g}_{\alpha}$ is one-dimensional, the $R$-module $L^+_{\alpha}$  is free of rank $1$; 
let $x_{\alpha}$ be a generator. 
Then the $R$-module $\msc{U}_{L^+,\chi}$ is generated by the finite set
\begin{equation*}
\left\{x_{\alpha_1} \cdot x_{\alpha_2} \cdots x_{\alpha_k}: k \in \mathbb{Z}_{\geq 0}, \sum_i \alpha_i = \chi\right\}.
\end{equation*}
On the other hand, the Poincaré--Birkhoff--Witt theorem tells us that the $K$-vector space $U(\mfr{n}^+)_{\chi}$ is also generated by this set; 
hence $\msc{U}_{L^+,\chi}$ is a split lattice in $U(\mfr{n}^+)_{\chi}$.
\item We start with $S^-(L^-,J)$. 
Since the action of $U(\mfr{n}^-)_{\chi}$ sends $V_{(\psi),\chi'}$ to $V_{(\psi),\chi+\chi'}$, we see that
\begin{equation*}
S^-(L^-,J) = \bigoplus_{\psi \in \mathcal{D}} \bigoplus_{\chi \in Q} \msc{U}_{L^-,\chi} \cdot J_{\psi} = \bigoplus_{\psi \in \mathcal{D}} \bigoplus_{\chi \in Q} S^-(L^-,J) \cap V_{(\psi),\psi+\chi},
\end{equation*}
hence $S^-(L^-,J)$ is split. Since $\msc{U}_{L^-,\chi}$ is a finitely generated $R$-module spanning $U(\mfr{n}^-)_{\chi}$, 
and $J^-_{\psi}$ is a finitely generated $R$-module spanning $V_{(\psi),\psi}$, 
we may conclude that $U(L^-)_{\chi} \cdot J^-_{\psi}$ is a finitely generated $R$-module spanning $U(\mfr{n}^-)_{\chi} \cdot V_{(\psi),\psi}$, 
which is equal to $V_{(\psi),\psi+\chi}$ by Proposition \ref{prop:surj}. 
Hence $S^-(L^-,J)_{(\psi),\psi+\chi}$ is a lattice in $V_{(\psi),\psi+\chi}$, 
and since $S^-(L^-,J)$ is a split $R$-module, 
this shows that $S^-(L^-,J)$ is a lattice in $V$.

Now consider $S^+(L^+,J)$.
 Let $x \in V$, and write $x = \sum_{\psi,\chi} x_{(\psi),\chi}$ with $x_{(\psi),\chi} \in V_{(\psi),\chi}$. 
 Then for every $\psi \in \msc{D}$ we have
\begin{equation*}
\recht{pr}_{(\psi),\psi}(\msc{U}_{L^+} \cdot x) = \sum_{\chi \in Q} \recht{pr}_{(\psi),\psi}(\msc{U}_{L^+,\chi} \cdot x) = \sum_{\chi \in Q} \msc{U}_{L^+,\chi} \cdot x_{(\psi),\psi-\chi},
\end{equation*}
hence $x$ is an element of $S^+(L^+,J)$ if and only if $x_{(\psi),\chi}$ is for all $\psi \in \mathcal{D}$ and all $\chi \in \recht{X}^*(T)$; 
this shows that $S^+(L^+,J)$ is split with respect to the decomposition $V = \bigoplus_{\psi,\chi} V_{(\psi),\chi}$. 
We now need to show that $S^+(L^+,J)_{(\psi),\chi}$ is a lattice in $V_{(\psi),\chi}$. 
Fix a $\chi$ and $\psi$, and choose a basis $f_1,\ldots,f_k$ of $J_{\psi}$; 
then $W_i := U(\mfr{g}) \cdot f_i$ is an irreducible subrepresentation of $V_{(\psi)}$. 
We get a decomposition $V_{(\psi),\chi} = \bigoplus_i W_{i,\chi}$, and from the definition of $S^+(L^+,J)$ we get
\begin{equation*}
S^+(L^+,J)_{(\psi),\chi} = \bigoplus_i S^+(L^+,J)_{(\psi),\chi} \cap W_{i,\chi}.
\end{equation*}
As such, we need to show that for each $i$ the $R$-module $S_{i,\chi} := S^+(L^+,J)_{(\psi),\chi} \cap W_{i,\chi}$ is a lattice in $W_{i,\chi}$. 
Fix an $i$, and let $e_1,\ldots,e_n$ be a basis of $W_{i,\chi}$. 
For $j \leq n$, let $\varphi_j\colon W_{i,\chi} \rightarrow W_{i,\psi} = K \cdot f_i$ be the linear map that sends $e_j$ to $f_i$, and the other $e_{j'}$ to $0$. 
By Proposition \ref{prop:surj} there exists a $u_j \in U(\mfr{n}^+)$ such that $u_j$ acts like $\varphi_j$ on $W_{i,\chi}$. 
Since $U(L^+ )$ is a lattice in $U(\mfr{n}^+)$, there exists a $r \in R$ such that $ru_j \in \msc{U}_{L^+}$ for all $j$. 
Then for all $x \in S_{i,\chi}$ one has $ru_j \cdot x \in Rf_i$ for all $j$, 
so $x$ lies in the free $R$-submodule of $W_{i,\chi}$ generated by $r^{-1}e_1,\ldots,r^{-1}e_n$; 
hence $S_{i,\chi}$ is finitely generated.

Now we need to show that $S_{i,\chi}$ spans $W_{i,\chi}$. 
Since $\msc{U}_{L^+,\psi-\chi}$ is finitely generated, for every $z \in W_{i,\chi}$ we get that $\msc{U}_{L^+,\psi-\chi} \cdot z$ is a lattice in $W_{i,\psi} = K\cdot f_i$.
As such we can find some $r' \in R$ such that $\msc{U}_{L^+,\psi-\chi} \cdot r'z \subset R \cdot f_i \subset J_{\psi}$;
hence $r'z \in S_{i,\chi}$. 
This shows that $S^+(L^+,J)_{(\psi),\chi,i}$ generates $W_{i,\chi}$ as a $K$-vector space, 
so $S_{i,\chi}$ is a lattice in $W_{i,\chi}$, as was to be shown.

\item Since $\msc{U}_{L^+,0} = \msc{U}_{L^-,0} = R$ we immediately get $S^+(L^+,J)_{(\psi),\psi} = J_{\psi}$ for all $\psi$. 
The other statement follows immediately from the definition of $S^+(L^+,J)$ and $S^-(L^-,J)$. \qedhere
\end{enumerate}
\end{proof}

\begin{rem} By Proposition \ref{prop:spm} we can define maps $S^{\pm}\colon \mathcal{L}^{\pm} \times \mathcal{J} \rightarrow \recht{Lat}(V)$.
\end{rem}

Let $H = Z \cdot T$ as before. 
Since $H$ normalises $G$, we see that $H$ acts on $G$ by conjugation. 
This gives us a representation $\varrho\colon H \rightarrow \mathtt{GL}(\mfr{g})$. 
Since $Z$ acts trivially on $G$, we see that the image of $H$ in $\recht{GL}(\mfr{g})$ is equal to $\bar{T}$;
hence the action of $H$ on $\mfr{g}$ respects the decomposition $\mfr{g} = \mfr{t} \oplus \bigoplus_{\alpha \in \Psi} \mfr{g}_{\alpha}$.

\begin{lem} \label{lem:hbart}
The map $\varrho\colon H \twoheadrightarrow \bar{T}$ is surjective on $K$-points.
\end{lem}

\begin{proof} The short exact sequence $1 \rightarrow Z \rightarrow H \rightarrow \bar{T} \rightarrow 1$ induces a longer exact sequence
\begin{equation*}
1 \rightarrow Z(K) \rightarrow H(K) \rightarrow \bar{T}(K) \rightarrow \recht{H}^1(K,Z).
\end{equation*}
Since $\recht{H}^1(K,\mathtt{GL}_n)$ is trivial for every integer $n$ and $Z$ is isomorphic to a product of $\mathtt{GL}_n$s by Remark \ref{rem:not}.2, 
this implies that the map $H(K) \rightarrow \bar{T}(K)$ is surjective.
\end{proof}

Since the action of $H$ on $\mfr{g}$ respects its decomposition into root spaces, 
we get an action of $H(K)$ on the sets $\mathcal{L}^{\pm}$. 
Furthermore, the representation $H \hookrightarrow \mathtt{GL}(V)$ respects the decomposition $V = \bigoplus_{\psi \in \mathcal{D}} V_{(\psi)}$. 
Since $H$ normalises $T$, the action of $H$ also respects the decomposition $V_{(\psi)} = \bigoplus_{\chi \in \recht{X}^*(T)} V_{(\psi),\chi}$; 
hence $H(K)$ acts on the set $\mathcal{J}$.

\begin{prop} \label{prop:trans}
The maps $S^{\pm}\colon \mathcal{L}^{\pm} \times \mathcal{J} \rightarrow \recht{Lat}(V)$ are $H(K)$-equivariant, and the action of $H(K)$ on $\mathcal{L}^{\pm} \times \mathcal{J}$ is transitive.
\end{prop}

\begin{proof}
The Lie algebra action map $\mfr{g} \times V \rightarrow V$ is equivariant with respect to the action of $H(K)$ on both sides. 
From the definition of $S^{\pm}(L^{\pm},J)$ it now follows that $S^{\pm}(h \cdot L^{\pm},h \cdot J) = h \cdot S^{\pm}(L^{\pm},J)$ for all $h \in H(K)$.
Let $L_1^+, L_2^+ \in \mathcal{L}^+$ and $J_1,J_2 \in \mathcal{J}$. 
For every $\alpha \in \Delta^+$, let $x_{\alpha} \in K^{\times}$ be such that $L^+_{1,\alpha} = x_{\alpha}L^+_{2,\alpha}$; 
the scalar $x_{\alpha}$ exists because $L^+_{1,\alpha}$ and $L^+_{2,\alpha}$ are free lattices in the same one-dimensional vector space. 
Since $\Delta^+$ is a basis for $Q = \recht{X}^*(\bar{T})$ (see Remark \ref{rem:not}.1) there exists a unique $t \in \bar{T}(K)$ such that $\alpha(t) = x_{\alpha}$ for all $\alpha \in \Delta^+$. 
By Lemma \ref{lem:hbart} there exists an $h \in H(K)$ such that $\varrho(h) = t$; 
then $h \cdot L_1^+ = L_2^+$. 
Since $Z(K)$ acts transitively on $\mathcal{J}$ by remark \ref{rem:not}.2, there exists a $z \in Z(K)$ such that $z \cdot (h \cdot J_1) = J_2$. 
As $z$ acts trivially on $\mathcal{L}^+$, we get $zh \cdot (L_1^+,J_1) = (L_2^+,J_2)$; 
this shows that $H(K)$ acts transitively on $\msc{L}^+ \times \msc{J}$. 
The proof for $\msc{L}^-$ is analogous.
\end{proof}

\subsection{Chevalley lattices} \label{ss:chevlat}

In this subsection we consider lattices in the $K$-vector space $\mfr{g}$. 
We will define the set of Chevalley lattices in $\mfr{g}$. 
The distance (in the sense of Lemma \ref{lem:dist}) between such a Chevalley lattice and the lattice corresponding to a model of $(G,T)$ will serve as a good measure of the `ugliness' of the model, and this will allow us to prove finiteness results. 

Let $G^{\recht{der}}$ be the derived group of $G$, and let $T'$ be the identity component of $T \cap G^{\recht{der}}$. 
Let $\mfr{g}^{\recht{ss}}$ and $\mfr{t}'$ be the Lie algebras of $G^{\recht{der}}$ and $T'$, respectively. 
The roots of $G$ (with respect to $T$) induce linear maps $\recht{Lie}(\alpha)\colon \mfr{t}' \rightarrow K$, and these form the root system of the split semisimple Lie algebra $(\mfr{g}^{\recht{ss}},\mfr{t}')$.
Since the Killing form $\kappa$ on $\mfr{t}'$ is nondegenerate 
there exists a unique $t_{\alpha} \in \mfr{t}'$ such that $\kappa(t_{\alpha},-) = \recht{Lie}(\alpha)$. Since $\kappa(t_{\alpha},t_{\alpha}) \neq 0$ we may define $h_{\alpha} := \frac{2}{\kappa(t_{\alpha},t_{\alpha})}t_{\alpha}$; see \cite[Prop.~8.3]{humphreys1972}.

\begin{defn}
An element $x = (x_{\alpha})_{\alpha \in \Psi}$ of $\prod_{\alpha \in \Psi} (\mfr{g}_{\alpha} \backslash \{0\})$ is called a \emph{Chevalley set} if the following conditions are satisfied:
\begin{enumerate}
\item $[x_{\alpha},x_{-\alpha}] = h_{\alpha}$ for all $\alpha \in \Psi$;
\item If $\alpha$ and $\beta$ are two $\mathbb{R}$-linearly independent roots such that $\beta+\mathbb{Z}\alpha$ intersects $\Psi$ in the elements $\beta-r\alpha,\beta-(r-1)\alpha,\ldots,\beta+q\alpha$, then $[x_{\alpha},x_{\beta}] = 0$ if $q = 0$, and $[x_{\alpha},x_{\beta}] = \pm (r+1)x_{\alpha+\beta}$ if $q > 0$.
\end{enumerate}
\end{defn}

There is a canonical isomorphism of $K$-vector spaces:
\begin{align} \label{eq:isomchar}
K \otimes_{\mathbb{Z}} \recht{X}^*(T) &\DistTo \mathfrak{t}^{\vee} \\
1 \otimes \alpha &\mapsto \recht{Lie}(\alpha).
\end{align}
Under this isomorphism, we can consider $\mathfrak{T}_0 := (R \otimes_{\mathbb{Z}} \recht{X}^*(T))^{\vee}$ as an $R$-submodule of $\mathfrak{t}$.

\begin{lem}
Let $\alpha \in \Phi$. Then $h_{\alpha} \in \mathfrak{T}_0$.
\end{lem}

\begin{proof}
It suffices to show that $\recht{Lie}(\lambda)(h_{\alpha}) \in \mathbb{Z}$ for all $\lambda \in \recht{X}^*(T)$. 
Since the action of $\lambda \in \recht{X}^*(T)$ on $\mfr{t}'$ only depends on its image in $\recht{X}^*(T')$, it suffices to prove this for semisimple $G$, for which see \cite[{}31.1]{humphreys1975}.
\end{proof}

\begin{defn} \label{def:chevlat}
A \emph{Chevalley lattice} is an $R$-submodule of $\mfr{g}$ of the form
\begin{equation*}
\mfr{C}(x) = \mfr{T}_0 \oplus \bigoplus_{\alpha \in \Psi} R \cdot x_{\alpha},
\end{equation*}
where $x$ is a Chevalley set. 
The set of Chevalley lattices is denoted $\mathcal{C}$.
\end{defn}

\begin{rem}
It is clear that $\mfr{C}(x)$ is a finitely generated $R$-submodule of $\mfr{g}$ that generates $\mfr{g}$ as a $K$-vector space, hence it is indeed a lattice.
The name comes from the fact that if $G$ is adjoint, then $\{h_{\alpha}\}_{\alpha \in \Delta^+} \cup \{x_{\alpha}\colon \alpha \in \Psi\}$ is a \emph{Chevalley basis} of $\mfr{g}$ in the sense of \cite[{}25.2]{humphreys1972}, and the Lie algebra of the Chevalley model (for any anchoring) is a Chevalley lattice in $\mfr{g}$.
\end{rem}

\begin{lem} \label{lem:chevexists}

 Let $\recht{Aut}(G,T) = \{\sigma \in \recht{Aut}(G): \sigma(T) = T\}$.
\begin{enumerate}
\item There exists a Chevalley lattice in $\mfr{g}$.
\item Every Chevalley lattice is an $R$-Lie subalgebra of $\mfr{g}$.
\item Let $\sigma \in \recht{Aut}(G,T)$, and let $\mfr{C} \in \mathcal{C}$. Then the lattice $\sigma(\mfr{C}) \subset \mfr{g}$ is again a Chevalley lattice.
\end{enumerate}
\end{lem}

\begin{proof} 
\begin{enumerate}
\item It suffices to show that a Chevalley set exists, for which we refer to \cite[Th.~25.2]{humphreys1972}.
\item By definition we have $[x_{\alpha},x_{-\alpha}] \in \mfr{T}_0$ and $[x_{\alpha},x_{\beta}] \in R\cdot x_{\alpha+\beta}$ if $\alpha+\beta \neq 0$.
 Furthermore for $t \in \mathfrak{T}_0$ one has $[t,x_{\alpha}] = \recht{Lie}(\alpha)(t) \cdot x_{\alpha} \in R \cdot x_{\alpha}$ by definition of $\mathfrak{T}_0$.
\item The automorphism $\sigma \in \recht{Aut}(G,T)$ induces an automorphism $\bar{\sigma}$ of $\Psi$. 
Then $\sigma$ maps $\mfr{g}_{\alpha}$ to $\mfr{g}_{\bar{\sigma}(\alpha)}$ and $\mfr{T}_0$ to itself. 
Let $x$ be a Chevalley set such that $\mfr{C} = \mfr{C}(x)$, and define $x' = (x'_{\alpha})_{\alpha \in \Psi}$ by $x'_{\alpha} = \sigma(x_{\bar{\sigma}^{-1}(\alpha)})$. 
Since $\sigma(h_{\alpha}) = h_{\bar{\sigma}(\alpha)}$ this is again a Chevalley set, and $\sigma(\mfr{C}) = \mfr{C}(x')$. \qedhere
\end{enumerate}
\end{proof}

It is easily checked that the action of $H(K)$ on $\recht{Lat}(\mfr{g})$ sends its subset $\mathcal{C}$ to itself. 
Furthermore there are natural isomorphisms of $H(K)$-sets
\begin{align} \label{eq:fpm}
f^{\pm}\colon \mathcal{C} &\DistTo \mathcal{L}^{\pm} \\
\mfr{C} & \mapsto (\mfr{C} \cap \mfr{g}_{\pm \alpha})_{\alpha \in \Delta^+}.
\end{align}
Since the action of $H(K)$ on $\msc{L}^{\pm}$ is transitive, we have shown:
\begin{lem}
The action of $H(K)$ on $\mathcal{C}$ is transitive. \qed
\end{lem}

\begin{lem} \label{lem:nulinv}
Let $\mfr{C} \in \mathcal{C}$ be a Chevalley lattice and let $\msc{U}_{\mfr{C}}$ be the $R$-subalgebra of $U(\mfr{g})$ generated by $\mfr{C}$. 
Then $\msc{U}_{\mfr{C}}$ is split with respect to the $Q$-grading of $U(\mfr{g})$. 
The subalgebra $\msc{U}_{\mfr{C},0} \subset U(\mfr{g})_0$ does not depend on the choice of $\mfr{C}$.
\end{lem}

\begin{proof}
The fact that $\msc{U}_{\mfr{C}}$ is split follows from the fact that it is generated by elements of pure degree. 
Now let $\mfr{C},\mfr{C}' \in \mathcal{C}$. 
Since $H(K)$ acts transitively on $\mathcal{C}$ and the action of $H$ on $\mathcal{C}$ factors through $\bar{T}$, there exists a $t \in \bar{T}(K)$ such that $t\cdot\mfr{C} = \mfr{C}'$. 
Then $\msc{U}_{\mfr{C}'} = t \cdot \msc{U}_{\mfr{C}'}$, where $t$ acts on $U(\mfr{g})$ according to its $Q$-grading. 
In particular this shows that $\msc{U}_{\mfr{C},0} = \msc{U}_{\mfr{C}',0}$.
\end{proof}

\begin{lem} \label{lem:projchev}
There exists an $r \in R$ such that for every Chevalley lattice $\mfr{C}$, every $\psi \in \mathcal{D}$ and every $\chi \in \recht{X}^*(T)$, the endomorphism $r\recht{pr}_{(\psi),\chi}$ of $V$ lies in the image of the map $\msc{U}_{\mfr{C}} \rightarrow \recht{End}(V)$.
\end{lem}

\begin{proof}
First, we show that $\recht{pr}_{(\psi),\chi}$ lies in the image of $U(\mfr{g})$.
Fix a $\psi_0 \in \mathcal{D}$ and a $\chi \in \recht{X}^*(T)$. 
For every $\psi \in \mathcal{D}$, let $W(\psi)$ be the irreducible representation of $G$ of highest weight $\psi$. 
Let $f \in \bigoplus_{\psi \in \mathcal{D}} \recht{End}(W(\psi))$ be the element whose $\psi_0$-component is $\recht{pr}_{\chi}$ and whose other components are $0$. 
By Theorem \ref{thm:jacobson} there exists a $u_{\psi_0,\chi} \in U(\mfr{g})_0$ that acts as $f$ on $\bigoplus_{\psi \in \msc{D}} W(\psi)$; 
then $u_{\psi_0,\chi}$ acts as $\recht{pr}_{(\psi),\chi}$ on $V$. 
Let $\mfr{C}$ be a Chevalley lattice. Since there are only finitely many choices for $\psi$ and $\chi$, there exists an $r \in R$ such that $ru_{\psi,\chi} \in  \msc{U}_{\mfr{C},0}$ for all $\psi \in \mathcal{D}$ and all $\chi$ for which $V_{(\psi),\chi} \neq 0$. 
Then $r$ satisfies the properties of the Lemma for $\mfr{C}$. By Lemma \ref{lem:nulinv} the element $r$ works regardless of the choice of $\mfr{C}$, which proves the Lemma.
\end{proof}

\subsection{Chevalley-invariant lattices}

In this section we consider lattices in $V$ that are invariant under some Chevalley lattice in $\mfr{g}$. 
The main result is that for $p$-adic $K$, up to $H(K)$-action only finitely many such lattices exist (Proposition \ref{prop:finhorbloc}).

\begin{defn} \label{def:chevinvlat}
Let $\Lambda$ be a lattice in $V$. We call $\Lambda$ \emph{Chevalley-invariant} if there exists a Chevalley lattice $\mfr{C} \subset \mfr{g}$ such that $\mfr{C} \cdot \Lambda \subset \Lambda$.
\end{defn}

\begin{lem} \label{lem:chev}
There exists a Chevalley-invariant lattice in $V$.
\end{lem}

\begin{proof}
This is proven for $K = \mathbb{Q}$ in \cite[Ch.~VIII, §12.8, Th.~4]{bourbaki2006};
the proof given there also works for general $K$. 
\end{proof}

\begin{lem}
The set of Chevalley-invariant lattices is invariant under the action of $H(K)$ on $V$.
\end{lem}

\begin{proof}
If $\Lambda$ is invariant under a Chevalley lattice $\mfr{C}$ and $h$ is an element of $H(K)$, then $h \cdot \Lambda$ is invariant under $h \cdot \mfr{C}$; 
hence this follows from the fact that the set of Chevalley lattices is invariant under the action of $H(K)$.
\end{proof}

\begin{rem}
Since $H(K)$ acts transitively on the set of Chevalley lattices, we see that for every Chevalley lattice $\mfr{C}$ there is a lattice in $V$ invariant under $\mfr{C}$.
\end{rem}

The following Lemma shows that every Chevalley-invariant lattice is contained within the lattices $S^{\pm}(L^{\pm},J)$ introduced in section \ref{ss:lat}. We let $f^{\pm}$ be the maps from (\ref{eq:fpm}).

\begin{lem} \label{lem:chevbound}
Let $\mfr{C}_0$ be a Chevalley lattice in $\mfr{g}$, and let $J_0 \in \mathcal{J}$. 
Let $L_0^+ = f^+(\mfr{C}_0)$ and $L_0^- = f^-(\mfr{C}_0)$. 
Let $\Lambda \subset V$ be a split Chevalley-invariant lattice such that $\Lambda_{(\psi),\psi}$ is a free $R$-module for all $\psi \in \mathcal{D}$. Then there exists an $h \in H(K)$ such that
\begin{equation*}
S^-(L_0^-,J_0) \subset h \cdot \Lambda \subset S^+(L_0^+,J_0).
\end{equation*}
\end{lem}

\begin{proof}
Let $\mfr{C}$ be a Chevalley lattice in $\mfr{g}$ such that $\mfr{C} \cdot \Lambda \subset \Lambda$. 
Let $J = (\Lambda_{(\psi),\psi})_{\psi \in \mathcal{D}}$; 
since all $\Lambda_{(\psi),\psi)}$ are free lattices, it is an element of $\mathcal{J}$.
Since $\mathcal{C}$ is isomorphic to $\msc{L}^+$ as $H(K)$-sets, by Proposition \ref{prop:trans} there exists an $h \in H(K)$ such that $h \cdot \mfr{C} = \mfr{C}_0$ and $h \cdot J = J_0$. 
Now let $\Lambda_0 = h \cdot \Lambda$; 
this is a split lattice satisfying $(\Lambda_0)_{(\psi),\psi} = J_{0,\psi}$ for all $\psi$. 
Furthermore, the lattice $\Lambda_0$ is invariant under the action of the Chevalley lattice $\mfr{C}_0$; 
in particular it is invariant under the action of the $f^+(\mfr{C}_0)_{\alpha} = \mfr{C}_0 \cap \mfr{g}_{\alpha}$ and the $f^-(\mfr{C}_0)_{\alpha} = \mfr{C}_0 \cap \mfr{g}_{-\alpha}$. 
By Proposition \ref{prop:spm}.3 we now get
\[
S^-(L_0^-,J_0) \subset \Lambda_0 \subset S^+(L_0^+,J_0). \qedhere
\]
\end{proof}

\begin{prop} \label{prop:finhorbloc}
Suppose $K$ is a $p$-adic field. 
Then there are only finitely many $H(K)$-orbits of Chevalley-invariant lattices.
\end{prop}

\begin{proof}
Let $\mfr{C}_0$, $J_0$, $L^+_0$ and $L^-_0$ be as in the previous Lemma. 
Let $\omega$ be a uniformiser of $K$. 
Let $m \in \mathbb{Z}_{\geq 0}$ be such that $\omega^mS^+(L_0^+,J_0) \subset S^-(L_0^-,J_0)$, and let $n \in \mathbb{Z}_{\geq 0}$ be such that for every Chevalley lattice $\mfr{C}$, every $\psi \in \mathcal{D}$ and every $\chi \in \recht{X}^*(T)$ the endomorphism $\omega^n\recht{pr}_{(\psi),\chi} \in \recht{End}(V)$ lies in the image of $\msc{U}_{\mfr{C}}$;
such an $n$ exists by Lemma \ref{lem:projchev}. 
Let $P^+$ be the $H(K)$-orbit of lattices of the form $S^+(L^+,J)$ (see Proposition \ref{prop:trans}). 
Let $X$ be an $H(K)$-orbit of Chevalley-invariant lattices. 
Let $\Lambda$ be an element of $X$, and let $\mfr{C}$ be a Chevalley lattice such that $\Lambda$ is invariant under $\mfr{C}$. 
Then $\Lambda$ is invariant under the action of $\msc{U}_{\mfr{C}}$, hence
\begin{equation} \label{eq:finhorbloc}
\omega^n \bigoplus_{(\psi),\chi} \recht{pr}_{(\psi),\chi}\Lambda \subset \msc{U}_{\mfr{C}} \cdot \Lambda = \Lambda \subset \bigoplus_{(\psi),\chi} \recht{pr}_{(\psi),\chi}\Lambda.
\end{equation}
Since $\mfr{C} = \bigoplus_{\chi} \mfr{C}_{\chi}$, we see that $\Lambda' := \bigoplus_{(\psi),\chi} \recht{pr}_{(\psi),\chi}\Lambda$ is also closed under multiplication by $\mfr{C}$. 
Then (\ref{eq:finhorbloc}) tells that $d(\Lambda,\Lambda') \leq n$ (using the function $d$ from section \ref{ss:dist}). 
Since $K$ is a $p$-adic field, all locally free $R$-modules are in fact free, hence $\Lambda'$ satisfies the conditions of Lemma \ref{lem:chevbound}, and there exists an $h \in H(K)$ such that
\begin{equation*}
S^-(L_0^-,J_0) \subset h \cdot \Lambda' \subset S^+(L_0^+,J_0).
\end{equation*}
It follows from this that $d(h \cdot \Lambda',S^+(L_0^+,J_0)) \leq m$. From this we find
\begin{align*}
d_H(X,P^+) &\leq d_H(X,H(K) \cdot \Lambda') + d_H(H(K) \cdot \Lambda',P^+) \\
&\leq d(\Lambda,\Lambda') +  d(h \cdot \Lambda',S^+(L_0^+,J_0)) \\
&\leq n+m.
\end{align*}
This shows that all $H(K)$-orbits of Chevalley-invariant lattices lie within a ball of radius $n+m$ around $P^+$ in the metric space $(\recht{Lat}_H(V),d_H)$.
By Lemma \ref{lem:dist}.3 this ball is finite, which proves the Proposition.
\end{proof}

\begin{prop} \label{prop:finhorbglob}
Let $K$ be a number field. 
Then for almost all finite places $v$ of $K$ there is exactly one $H(K_v)$-orbit of Chevalley-invariant lattices in $\recht{Lat}(V_{K_v})$.
\end{prop}

\begin{proof}
Fix a Chevalley lattice $\mfr{C} \subset \mfr{g}$ and a $J \in \mathcal{J}$, and let $L^{\pm} = f^{\pm}(\mfr{C})$. 
For a finite place $v$ of $K$ write $\mfr{C}_v := \mfr{C}_{R_v}$ and $J_v = (J_{\psi,R_v})_{\psi \in \mathcal{D}}$. 
Then $\mfr{C}_v$ is a Chevalley lattice in $\mfr{g}_{K_v}$, and we set $L_v^{\pm} := f^{\pm}(\mfr{C}_v)$; 
then it follows from the definitions of $f^{\pm}$ and $S^{\pm}$ that $L_v^{\pm} = (L^{\pm}_{\alpha,R_v})_{\alpha \in \Delta^+}$ and
\begin{equation*}
S^{\pm}(L^{\pm}_v,J_v) = S^{\pm}(L^{\pm},J)_{R_v} \subset V_{K_v}.
\end{equation*}
This shows that $S^{-}(L^{-}_v,J_v) = S^{+}(L^{+}_v,J_v)$ for almost all $v$. 
Furthermore, let $r$ be as in Lemma \ref{lem:projchev}; 
then $v(r) = 0$ for almost all $v$. 
Now let $v$ be such that $S^{-}(L^{-}_v,J_v) = S^{+}(L^{+}_v,J_v)$ and $v(r) = 0$. 
Consider the proof of the previous Proposition, applied to the group $G_{K_v}$ and its representation on $V_{K_v}$, taking $\mfr{C}_0 := \mfr{C}_v$ and $J_0 := J_v$. 
In the notation of that proof we get $m=n=0$, hence $X = P^+$, and there is exactly one orbit of Chevalley-invariant lattices.
\end{proof}

\subsection{Models of split reductive groups}

In this section we apply our results about lattices in representations of Lie algebras to prove Theorem \ref{thm:split}. 
The strategy is to give a bound for the distance between a lattice $\Lambda$ and a Chevalley-invariant lattice in $V$, in terms of the distance between the Lie algebra of $\mathtt{mdl}_G(\Lambda)$ and a Chevalley lattice in $\mfr{g}$. 
Combined with Propositions \ref{prop:finhorbloc} and \ref{prop:finhorbglob} this will give the desired finiteness properties.

\begin{nota}
For the rest of this section, we write $\varrho: U(\mfr{g}) \rightarrow \recht{End}(V)$ for the homomorphism of $K$ algebras induced by the representation $\mfr{g} \rightarrow \mfr{gl}(V)$.
Furthermore, if $(\msc{G},\msc{T})$ be a model of $(G,T)$, $\mfr{G}$ is the Lie algebra of $\msc{G}$, and $\varphi$ is an anchoring of $(\msc{G},\msc{T})$, then we write $\msc{U}_{\mfr{G},\varphi}$ for the $R$-subalgebra of $U(\mfr{g})$ generated by $(\recht{Lie} \ \varphi)(\mfr{G}) \subset \mfr{g}$. We write $\msc{V}_{\mfr{G},\varphi} := \varrho(\msc{U}_{\mfr{G},\varphi})$ and $\msc{V}_{\mfr{C}} := \varrho(\msc{U}_{\mfr{C}})$ for a Chevallley lattice $\mfr{C}$, where $\msc{U}_{\mfr{C}}$ is as in Lemma \ref{lem:nulinv}.
\end{nota}

\begin{lem} \label{lem:ulat}
\begin{enumerate}
\item Let $\Lambda$ be a lattice in $V$, let $(\msc{G},\varphi) = \mathtt{mdl}^{\recht{a}}_G(\Lambda)$ be the anchored model of $G$ associated to $\Lambda$, and let $\mfr{G}$ be the Lie algebra of $\msc{G}$.  
Then $\msc{V}_{\mfr{G},\varphi}$ is a lattice in the $K$-vector space $\varrho(U(\mfr{g}))$.
\item Let $\mfr{C} \subset \mfr{g}$ be a Chevalley lattice.
Then $\msc{V}_{\mfr{C}}$ is a lattice in $\varrho(U(\mfr{g}))$.
\end{enumerate}
\end{lem}

\begin{proof}
\begin{enumerate}
\item The image of $\msc{U}_{\mfr{G},\varphi}$ under $\varrho$ is contained in $\varrho(U(\mfr{g})) \cap \recht{End}(\Lambda)$. 
Since $\recht{End}(\Lambda)$ is a lattice in $\recht{End}(V)$, we see that $\varrho(U(\mfr{g})) \cap \recht{End}(\Lambda)$ is a lattice in $\varrho(U(\mfr{g}))$; 
hence $\msc{V}_{\mfr{G},\varphi}$ is finitely generated. 
On the other hand, $\msc{U}_{\mfr{G},\varphi}$ generates $U(\mfr{g})$ as a $K$-vector space, hence $\msc{V}_{\mfr{G},\varphi}$ spans $\varrho(U(\mfr{g}))$;
 hence it is a lattice in this vector space.
\item Let $\Lambda$ be a lattice invariant under $\mfr{C}$, and let $(\msc{G},\varphi)$ be its associated anchored model of $G$.
Then $\msc{V}_{\mfr{C}}$ is an $R$-submodule of the lattice $\msc{V}_{\mfr{G},\varphi}$ that generates $\varrho(U(\mfr{g}))$ as a $K$-vector space, i.e. a lattice in $\varrho(U(\mfr{g}))$. \qedhere
\end{enumerate}
\end{proof}

\begin{lem} \label{lem:anch}
Let $(\msc{G},\msc{T})$ be a model of $(G,T)$, and let $\mfr{G}$ be the Lie algebra of $\msc{G}$.
Then there is an $r \in R$ such that the following holds:
Let $\varphi$ be any anchoring of $(\msc{G},\msc{T})$ such that  $\msc{V}_{\mfr{G},\varphi}$ is a lattice in $\varrho(U(\mfr{g}))$.
Then there exists a Chevalley lattice $\mfr{C}$ such that $r\msc{V}_{\mfr{C}} \subset \msc{V}_{\mfr{G},\varphi}$.
\end{lem}

The proof of this Lemma consists of three steps: first we prove it in the case that $G$ is a torus, then for semisimple $G$, and then we combine these results to prove the statement for general reductive $G$.

\begin{lem}
Lemma \ref{lem:anch} holds if $G$ is a torus.
\end{lem}
\begin{proof}
There is only one Chevalley lattice, namely the $R$-Lie algebra $\mfr{T}^0$ from section \ref{ss:chevlat}.
The composite map $\mfr{g} \hookrightarrow U(\mfr{g}) \twoheadrightarrow \varrho(U(\mfr{g}))$ is a bijection, and $(\recht{Lie} \ \varphi)(\mfr{G})$ maps to $\msc{V}_{\mfr{G},\varphi}$.
By (\ref{eq:isomchar}) we have $\mfr{g} \cong K \otimes_{\mathbb{Z}} \recht{X}_*(G)$.
Under this identification a $\mu \in \recht{X}_*(G)$ acts on $V_{\chi}$ (for $\chi \in \recht{X}^*(G)$) as $\langle \mu,\chi \rangle$.
Let $M$ be a basis of $\recht{X}_*(G)$, and let $k = \sum_{\mu \in M} k_{\mu} \mu$ be an element of $\mfr{g}$.
If there is a lattice $\Lambda \subset V$ such that $k \cdot \Lambda \subset \Lambda$, then $\sum_{\mu \in M} k_{\mu} \langle \mu, \chi \rangle \in R$ for all $\chi$ present in $V$.
Since $V$ is a faithful representation of $G$, these $\chi$ span $\recht{X}^*(G)$, hence $k_{\mu} \in R$ for all $\mu$, and $k \in \mfr{T}^0$.
Now let $\varphi$ be such that $\msc{V}_{\mfr{G},\varphi}$ is a lattice in $\varrho(U(\mfr{g}))$; then for any lattice $\Lambda \subset V$ the set $\msc{V}_{\mfr{G},\varphi} \cdot \Lambda \subset V$ is a lattice in $V$ which is fixed under the action of $\msc{V}_{\mfr{G},\varphi}$.
As such we find that  $\msc{V}_{\mfr{G},\varphi} \subset \msc{V}_{\mfr{T}^0}$.
Since $\recht{Aut}(G)$is a group of automorphisms of $\mfr{g}$ that sends $\mfr{T}^0$ to itself and permutes the different $(\recht{Lie} \ \varphi)(\mfr{G}) = \msc{V}_{\mfr{G},\varphi}$ transitively, we find that the index $r$ of this inclusion does not depend on the choice of $\varphi$;
this $r$ satisfies the conditions of the Lemma.
\end{proof}

\begin{lem}
Lemma \ref{lem:anch} holds if $G$ is semisimple.
\end{lem}

\begin{proof}
Fix a Chevalley lattice $\mfr{C}$ and an anchoring $\varphi$ of $(\msc{G},\msc{T})$. 
By Lemma \ref{lem:ulat} $\msc{V}_{\mfr{C}}$ is a lattice in $\varrho(U(\mfr{g}))$, 
hence there exists an $r_{\varphi} \in R$ such that $r_{\varphi}\msc{V}_{\mfr{C}} \subset\msc{V}_{\mfr{G},\varphi}$.  
Let $\recht{Aut}(G,T) = \{\sigma \in \recht{Aut}(G): \sigma(T) = T\}$ as in Lemma \ref{lem:chev}, and let $\mathtt{Aut}(G,T)$ be the underlying $K$-group scheme. 
There is a short exact sequence of algebraic groups over $K$
\begin{equation*}
1 \rightarrow G^{\recht{ad}} \rightarrow \mathtt{Aut}(G) \rightarrow \Gamma \rightarrow 1
\end{equation*}
where $\Gamma$ is the automorphism group of the based root datum $(\Psi,\Delta^+)$.
Since $G$ is semisimple, we know that $\Gamma$ is a finite group. 
The kernel of the map $\mathtt{Aut}(G,T) \rightarrow \Gamma$ is the image of the scheme-theoretic normaliser $\mathtt{Norm}_G(T)$ in $G^{\recht{ad}}$; 
its identity component is $\bar{T}$. 
Since $\Gamma$ is finite and the index of $\bar{T}$ in $\recht{ker}(\mathtt{Aut}(G,T) \rightarrow \Gamma)$ is finite, we see that $\bar{T}(K)$ has finite index in $\recht{Aut}(G,T)$. 
Now let $\varphi'$ be another anchoring of $(\msc{G},\msc{T})$. 
There exists a unique $\sigma \in \recht{Aut}(G,T)$ such that $\varphi' = \sigma \circ \varphi$. 
The automorphism $\sigma$ also induces automorphisms of $\mfr{g}$ and $U(\mfr{g})$, which we will still denote by $\sigma$. 
Suppose $\sigma$ is an inner automorphism corresponding to a $t \in \bar{T}(K)$. 
Then $\sigma$ acts as $\chi(t)$ on $U(\mfr{g})_{\chi}$ for every $\chi \in Q$. 
Since $\varrho$ is a homomorphism of $\recht{X}^*(T)$-graded algebras, we get
\[
\msc{V}_{\mfr{G},\varphi'} = \varrho(\sigma(\msc{U}_{\mfr{G},\varphi})) = \varrho(t \cdot \msc{U}_{\mfr{G},\varphi}) = t \cdot \msc{V}_{\mfr{G},\varphi} \supset r_{\varphi} \cdot t \cdot \msc{V}_{\mfr{C}} = r_{\varphi} \msc{V}_{\sigma(\mfr{C})}.
\]
This shows that $r = r_{\varphi}$ has the property that for all $\varphi'$ in the $\bar{T}(K)$-orbit of $\varphi$, there exists a Chevalley lattice $\mfr{C}'$ such that $r\msc{V}_{\mfr{C}'} \subset \msc{V}_{\mfr{G},\varphi'}$. 
Since there are only finitely many such orbits, we can find an $r$ that works for all of them simultaneously.
\end{proof}

\begin{proof}[Proof of Lemma \ref{lem:anch}]
Let $Z$ be the connected component of the center of $G$, and let $\mfr{z}$ be its Lie algebra.
Similarly, let $\msc{Z} \subset \msc{G}$ be the Zariski closure of the connected component of the center of $\msc{G}_K$, and let $\mfr{Z}$ be its Lie algebra.
Then $\msc{Z}$ is a model of $Z$, and any anchoring $\varphi$ of $\msc{Z}$ induces the $R$-Lie subalgebra $(\recht{Lie} \ \varphi)(\mfr{Z}) \subset \mfr{z}$, which generates the $R$-subalgebra $\msc{U}_{\mfr{Z},\varphi}$ of $U(\mfr{z})$.
We let $G'$ be the derived group of $G$, and $T'$ the connected component of $T \cap G'$, and analogous to the above we define $R$-group schemes $\msc{G}'$ and $\msc{T}'$, Lie algebras $\mfr{g}'$, $\mfr{t}'$, $\mfr{G}'$ and $\mfr{T}'$, and the $R$-subalgebra $\msc{U}_{\mfr{G}',\varphi}$ of $U(\mfr{g}')$.
By our results for tori and semisimple groups, this means that there exist $r_1,r_2 \in R$ such that the following holds: 
For each anchoring $\varphi_1$ of $\msc{Z}$ and each anchoring $\varphi_2$ of $(\msc{G}',\msc{T}')$ for which $\varrho(\msc{U}_{\mfr{Z},\varphi_1})$ and $\varrho(\msc{U}_{\mfr{G}',\varphi_2})$ are lattices, there exist Chevalley lattices $\mfr{C}_1 \subset \mfr{z}$ and $\mfr{C}_2 \subset \mfr{g}'$ such that
\begin{align}
r_1\varrho(\msc{U}_{\mfr{C}_1}) \subset \varrho(\msc{U}_{\mfr{Z},\varphi_1}), \label{eq:ulem1} \\
r_2\varrho(\msc{U}_{\mfr{C}_2}) \subset \varrho(\msc{U}_{\mfr{G}',\varphi_2}). \label{eq:ulem2}
\end{align}

Now let $\varphi$ be an anchoring of $(\msc{G},\msc{T})$ such that $\msc{V}_{\mfr{G},\varphi}$ is a lattice. 
Then the restrictions of $\varphi$ to $\msc{Z}$ and to $\msc{G}'$ give anchorings of $\msc{Z}$ and $(\msc{G}',\msc{T}')$. 
For these anchorings the following holds:
\begin{align*}
(\recht{Lie} \ \varphi)(\mfr{G}') &= (\recht{Lie} \ \varphi)(\mfr{G}) \cap \mfr{g}',\\
(\recht{Lie} \ \varphi)(\mfr{Z}) &= (\recht{Lie} \ \varphi)(\mfr{G}) \cap \mfr{z}.
\end{align*}
It follows from this that $\recht{Span}_R(\msc{U}_{\mfr{Z},\varphi}\cdot \msc{U}_{\mfr{G}',\varphi}) \subset \msc{U}_{\mfr{G},\varphi}$, where $\recht{Span}_R(X)$ is to be understood as the $R$-submodule of $U(\mfr{g})$ generated by $A$ and $B$.
In particular this implies that both $\msc{V}_{\mfr{Z},\varphi}$ and $\msc{V}_{\mfr{G}',\varphi}$ are subalgebras of $\msc{V}_{\mfr{G},\varphi}$, hence finitely generated. 
This shows that they are lattices in $\varrho(U(\mfr{z}))$ and $\varrho(U(\mfr{g}'))$, respectively. 
As such, there exist Chevalley lattices $\mfr{C}_1 \subset \mfr{z}$ and $\mfr{C}_2 \subset \mfr{g}'$ such that (\ref{eq:ulem1}) and (\ref{eq:ulem2}) hold.
Then $\mfr{C}_1 \oplus \mfr{C}_2$ is a Chevalley lattice in $\mfr{g}$ and
\[
 r_1r_2 \msc{V}_{\mfr{C}_1 \oplus \mfr{C}_2} =  r_1r_2 \recht{Span}_R(\msc{V}_{\mfr{C}_1} \cdot \msc{V}_{\mfr{C}_2}) \subset  \recht{Span}_R(\msc{V}_{\mfr{Z},\varphi}\cdot \msc{V}_{\mfr{G}',\varphi}) \subset  \msc{V}_{\mfr{G},\varphi}.
\]
This shows that $r = r_1r_2$ satisfies the conditions of the Lemma, as it does not depend on the choice of $\varphi$.
\end{proof}

\begin{prop} \label{prop:torusboundloc}
Suppose $K$ is a $p$-adic field. Then the map $\mathtt{mdl}_{G,T}\colon \recht{Lat}_H(V) \rightarrow \recht{Mdl}(G,T)$ of Lemma \ref{lem:mdldef} is finite.
\end{prop}

\begin{proof}
Let $(\msc{G},\msc{T})$ be a model of $(G,T)$, and let $r$ be as in Lemma \ref{lem:anch}.
Let $\mathcal{P} \subset \recht{Lat}_H(V)$ be the set of $H(K)$-orbits of Chevalley-invariant lattices; 
this is a finite set by Proposition \ref{prop:finhorbloc}. 
Let $X$ be an $H(K)$-orbit of lattices in $V$ such that $\mathtt{mdl}_{G,T}(X) = (\msc{G},\msc{T})$. 
Let $\Lambda \in X$, and let $\varphi$ be the anchoring of $(\msc{G},\msc{T})$ induced by $\Lambda$. 
Then $\Lambda$ is invariant under the action of $\msc{V}_{\mfr{G},\varphi}$.
Let $\mfr{C}$ be a Chevalley lattice in $\mfr{g}$ such that $r\msc{V}_{\mfr{C}} \subset \msc{V}_{\mfr{G},\varphi}$, and let $\Lambda' = \msc{V}_{\mfr{C}} \cdot \Lambda \subset V$. 
Since $\msc{V}_{\mfr{C}}$ is a finitely generated submodule of $\recht{End}(V)$, we see that $\Lambda'$ is a lattice in $V$ containing $\Lambda$ that is invariant under $\mfr{C}$. Furthermore we see
\[
\Lambda \subset  \Lambda' = \msc{V}_{\mfr{C}}\Lambda \subset r\msc{V}_{\mfr{G},\varphi}\Lambda = r\Lambda,
\]
hence $d(\Lambda,\Lambda') \leq v(r)$, where $v$ is the valuation on $K$.
 For the metric space $\recht{Lat}_H(V)$ this implies that $X$ is at most distance $v(r)$ from an element of $\mathcal{P}$. 
 Since $\mathcal{P}$ is finite and balls are finite in this metric space, we see that there are only finitely many possibilities for $X$.
\end{proof}

\begin{lem} \label{lem:torusboundglob}
Suppose $K$ is a number field. 
Then for almost all finite places $v$ of $K$ there is exactly one $H(K_v)$-orbit $X$ of lattices in $V_{K_v}$ such that $\mathtt{mdl}_{G_{K_v},T_{K_v}}(X)$ is the Chevalley model of $(G_{K_v},T_{K_v})$.
\end{lem}

\begin{proof}
Let $(\msc{G},\msc{T})$ be the Chevalley model of $(G,T)$, let $\varphi$ be some anchoring of $(\msc{G},\msc{T})$, and let $\mfr{C} \subset \mfr{g}$ be a Chevalley lattice. 
Then $\mfr{G}_{\varphi,R_v} = \mfr{C}_{R_v}$ as lattices in $\mfr{g}_{K_v}$ for almost all finite places $v$ of $K$. 
Hence for these $v$, the Lie algebra of the Chevalley model of $(G_{K_v},T_{K_v})$ is a Chevalley lattice via the embedding induced by the anchoring $\varphi$. 
However, two anchorings differ by an automorphism in $\recht{Aut}(G_{K_v},T_{K_v})$. 
Since the action of $\recht{Aut}(G_{K_v},T_{K_v})$ on $\recht{Lat}(\mfr{g}_{K_v})$ sends Chevalley lattices to Chevalley lattices by Lemma \ref{lem:chev}.3, this means that for these $v$ the Lie algebra of the Chevalley model will be a Chevalley lattice with respect to every anchoring. 
For these $v$, a lattice in $V_{K_v}$ yielding the Chevalley model must be Chevalley-invariant.
On the other hand, by Proposition \ref{prop:finhorbglob} for almost all $v$ there is precisely one $H(K_v)$-orbit of Chevalley-invariant lattices;
combining this, we find that for almost all $v$ there is at most one $H(K_v)$-orbit of lattices yielding the Chevalley model.
For existence, note that any model of $G$ will be reductive on an open subset of $\recht{Spec}(R)$, and any model of $T$ will be a split torus on an open subset of $\recht{Spec}(R)$. 
This shows that any model of $(G,T)$ is isomorphic to the Chevalley model over almost all $R_v$. 
It follows that for almost all $v$ there is at least one lattice yielding the Chevalley model.
\end{proof}

\begin{proof}[Proof of Theorem \ref{thm:split}]
\begin{enumerate}
\item Let $\msc{G}$ be a given model of $G$. 
Let $T$ be a split maximal torus of $G$, and choose a subgroup scheme $\msc{T} \subset \msc{G}$ such that $(\msc{G},\msc{T})$ is a model of $(G,T)$.
Let $\Lambda'$ be a lattice in $V$ with model $\mathtt{mdl}_{G,T}(\Lambda') = (\msc{G}',\msc{T}')$ and associated anchoring $\varphi'\colon\msc{G}'_K \DistTo G$.
Suppose there exists an isomorphism $\psi\colon \msc{G} \DistTo \msc{G}'$. 
Then $\psi(\msc{T}_K)$ is a split maximal torus of $\msc{G}'_K$. Since all split maximal tori of a split reductive group are conjugate (see \cite[Thm.~15.2.6]{springer1998}), there exists a $g \in \msc{G}'(K)$ such that $\psi(\msc{T}_K) = g\msc{T}'_Kg^{-1}$. 
Then $\recht{inn}(g) \circ \psi$ is an isomorphism of models of $(G,T)$ between $(\msc{G},\msc{T})$ and $\mathtt{mdl}_{G,T}(\varphi(g)\Lambda')$. 
By Proposition \ref{prop:torusboundloc} there are only finitely many $H(K)$-orbits yielding $(\msc{G},\msc{T})$, 
so $g\Lambda'$ can only lie in finitely many $H(K)$-orbits; 
hence $\Lambda'$ can only lie in finitely many $(G \cdot H)(K)$-orbits. Since $G \cdot H$ is a subgroup of $N$, this shows that there are only finitely many $N(K)$-orbits in $\recht{Lat}(V)$ yielding the model $\msc{G}$ of $G$.
\item Let $T$ be a split maximal torus of $G$. 
By Lemma \ref{lem:torusboundglob} for almost all finite places $v$ of $K$ there exists exactly one $H(K_v)$-orbit $Y_v \subset \recht{Lat}(V_{K_v})$ yielding the Chevalley model of $(G_{K_v},T_{K_v})$; 
let $v$ be such a place. 
Repeating the proof of the previous point, we see that $g\Lambda'$ has to lie in $Y_v$, hence $\Lambda'$ has to lie in $(G \cdot H)(K_v) \cdot Y_v$, and in particular in the single $N(K_v)$-orbit $N(K_v) \cdot Y_v$. \qedhere
\end{enumerate}
\end{proof}

\section{Nonsplit reductive groups} \label{s:local}

The main goal of this section is to prove Theorem \ref{thm:main1} for local fields, as well as a stronger finiteness result related to Theorem \ref{thm:split}.2 needed to prove Theorem \ref{thm:main1} for number fields, for connected reductive groups that are not required to be split.
We will make use of some Bruhat--Tits theory to prove the key Lemma \ref{lem:qss}.

\subsection{Bruhat--Tits buildings}

In this subsection we give a very brief summary of the part of Bruhat--Tits theory that is revelant to our purposes; 
Bruhat--Tits theory will only play a role in the proof of Lemma \ref{lem:qss}.
If $\Delta$ is a simplicial complex, I denote its topological realisation by $|\Delta|$.

\begin{thm} \label{thm:buildings1} 
\emph{(See \cite[Cor.~2.1.6; Lem.~2.5.1; 2.5.2]{bruhattits1972}, \cite[{}2.2.1]{tits1979} and \cite[Th.~VI.3A]{brown1989}.)}
Let $G$ be a connected semisimple algebraic group over a $p$-adic field $K$. 
Then there exists a locally finite simplicial complex $\mathcal{I}(G,K)$ with the following properties:
\begin{enumerate}
\item $\mathcal{I}(G,K)$ has finite dimension;
\item Every simplex is contained in a simplex of dimension $\recht{dim}(\mathcal{I}(G,K))$, and these maximal simplices are called \emph{chambers};
\item There is an action of $G(K)$ on $\mathcal{I}(G,K)$ that induces a proper and continuous action on $|\mathcal{I}(G,K)|$, where $G(K)$ is endowed with the $p$-adic topology;
\item The stabilisers of points in $|\mathcal{I}(G,K)|$ are compact open subgroups of $G(K)$;
\item $G(K)$ acts transitively on the set of chambers of $\mathcal{I}(G,K)$;
\item There is a metric $d$ on $|\mathcal{I}(G,K)|$ invariant under the action of $G(K)$ that gives the same topology as its topological realisation. \qedhere
\end{enumerate}
\end{thm}

\begin{rem} \label{rem:btdiscrete}
Since the stabiliser of each point is an open subgroup of $G(K)$, the $G(K)$-orbits in $|\mathcal{I}(G,K)|$ are discrete subsets.
\end{rem}

\begin{cor} \label{cor:btcompact}
Let $G$ be a semisimple algebraic group over a $p$-adic field $K$, let $C \subset |\mathcal{I}(G,K)|$ be a chamber, let $\bar{C} \subset |\msc{I}(G,K)|$ be its closure, and let $r \in \mathbb{R}_{>0}$. 
Then the subset $V \subset |\mathcal{I}(G,K)|$ given by 
\begin{equation*}
V := \Big\{x \in |\mathcal{I}(G,K)|:d(x,\bar{C}) \leq r\Big\}
\end{equation*}
 is compact.
\end{cor}

\begin{proof}
Since the metric of $|\mathcal{I}(G,K)|$ is invariant under the action of $G(K)$ and $G(K)$ acts transitively on the set of chambers, we see that every chamber has the same size. 
Since $\mathcal{I}(G,K)$ is locally finite this means that $V$ will only meet finitely many chambers. The union of the closures of these chambers is compact, hence $V$, being a closed subset of this, is compact as well.
\end{proof}

\begin{thm} \label{thm:buildings2}
\emph{(See \cite[Prop.~2.4.6; Cor.~5.2.2; Cor.~5.2.8]{rousseau1977})}
Let $G$ be a connected semisimple algebraic group over a $p$-adic field $K$, and let $L/K$ be a finite Galois extension.
\begin{enumerate}
\item The simplicial complex $\mathcal{I}(G,L)$ has a natural action of $\recht{Gal}(L/K)$;
\item The map $G(L) \times \mathcal{I}(G,L) \rightarrow \mathcal{I}(G,L)$ that gives the $G(L)$-action on $\mathcal{I}(G,L)$ is $\recht{Gal}(L/K)$-equivariant;
\item There is a canonical inclusion $\mathcal{I}(G,K) \hookrightarrow \mathcal{I}(G,L)^{\recht{Gal}(L/K)}$, which allows us to view $\mathcal{I}(G,K)$ as a subcomplex of $\mathcal{I}(G,L)$;
\item There is an $r \in \mathbb{R}_{>0}$ such that for every $x \in |\mathcal{I}(G,L)|^{\recht{Gal}(L/K)}$ there exists an $y \in |\mathcal{I}(G,K)|$ such that $d(x,y) \leq r$.
\end{enumerate}
\end{thm}

\subsection{Compact open subgroups and quotients}

Let $G$ be an algebraic group over a $p$-adic field $K$, and let $L$ be a finite Galois extension of $K$. 
Let $U$ be a compact open subgroup of $G(L)$ that is invariant under the action of $\recht{Gal}(L/K)$. 
Then $G(L)/U$ inherits an action of $\recht{Gal}(L/K)$, and its set of invariants $(G(L)/U)^{\recht{Gal}(L/K)}$ has a left action of $G(K)$. The goal of this section is to show that the quotient $G(K)\backslash (G(L)/U)^{\recht{Gal}(L/K)}$ is finite for various choices of $G$, $K$, $L$ and $U$. 
We will also show that it has cardinality $1$ if we choose $U$ suitably `nice'.

\begin{nota}
Let $G$ be an algebraic group over a $p$-adic field $K$, let $L/K$ be a finite Galois extension over which $G$ splits, and let $U$ be a compact open subgroup of $G(L)$ (with respect to the $p$-adic topology) fixed under the action of $\recht{Gal}(L/K)$. 
Then we write $Q^{L/K}_G(U) := G(K) \backslash (G(L)/U)^{\recht{Gal}(L/K)}$.
\end{nota}

The next Lemma tells us that compact open subgroups often appear in the contexts relevant to us.

\begin{lem} \label{lem:compop}
\emph{(See \cite[p.~134]{platonovrapinchuk1994})}
Let $G$ be an algebraic group over a $p$-adic field $K$, and let $L$ be a finite field extension of $K$. 
Let $(\msc{G},\varphi)$ be an anchored model of $G$.
Then $\varphi(\msc{G}(\msc{O}_L))$ is a compact open subgroup of $G(L)$ with respect to the $p$-adic topology. \qed
\end{lem}

We will now prove that $Q_G^{L/K}(U)$ is finite for reductive $G$. 
To prove this we first prove it for tori and for semisimple groups, and then combine these results.

\begin{lem} \label{lem:oneu}
Let $G$ be an algebraic group over a $p$-adic field $K$, and let $L/K$ be a finite Galois extension over which $G$ splits. 
If $Q^{L/K}_G(U)$ is finite for some compact open Galois invariant $U \subset G(L)$, then it is finite for all such $U$.
\end{lem}

\begin{proof}
This follows from the fact that if $U$ and $U'$ are compact open Galois invariant subgroups of $G(L)$, then $U'' := U \cap U'$ is as well, and $U''$ has finite index in both $U$ and $U'$.
\end{proof}

\begin{lem} \label{lem:qtorus}
Let $T$ be a torus over a $p$-adic field $K$, and let $L$ be a finite Galois extension of $K$ over which $T$ splits. 
Let $U$ be a compact open subgroup of $T(L)$. Then $Q^{L/K}_T(U)$ is finite.
\end{lem}

\begin{proof}
Choose an isomorphism $\varphi\colon T_L \DistTo \mathbb{G}_{\recht{m},L}^d$. 
Then $T(L)$ has a unique maximal compact open subgroup, namely $U = \varphi^{-1}((\msc{O}_L^{\times})^d)$; 
by Lemma \ref{lem:oneu} it suffices to prove this Lemma for this $U$.
Let $f$ be the ramification index of $L/K$, and let $\omega$ be a uniformiser of $L$ such that $\omega^f \in K$. 
Now consider the homomorphism of abelian groups
\begin{align*}
F\colon \recht{X}_*(T) &\rightarrow T(L)/U \\
\eta &\mapsto \eta(\omega)U.
\end{align*}
For every cocharacter $\eta$ the subgroup $\eta(\msc{O}_L^\times)$ of $T(L)$ is contained in $U$. 
This implies that for all $\eta \in \recht{X}_*(T)$ and all $\pi \in \recht{Gal}(L/K)$ one has
\begin{align*}
F(\pi \cdot \eta) &= \pi(\eta(\pi^{-1}\omega)))U \\
&= \pi(\eta(\omega)) \cdot \pi \left(\eta\left(\frac{\pi^{-1}\omega}{\omega}\right)\right)U\\
&= \pi(F(\eta)),
\end{align*}
since $\frac{\pi^{-1}\omega}{\omega} \in \msc{O}_L^\times$. 
This shows that $F$ is Galois-equivariant. 
On the other hand $\varphi$ induces isomorphisms of abelian groups $\recht{X}_*(T) \cong \mathbb{Z}^d$ and $T(L)/U \cong (L^{\times}/\msc{O}_L^{\times})^d = \langle \omega \rangle^d$. In terms of these identifications the map $F$ is given by
\begin{equation*}
\mathbb{Z}^d \ni (x_1,\ldots,x_d) \mapsto (\omega^{x_1},\ldots,\omega^{x_d}) \in \langle \omega \rangle^d \cong T(L)/U.
\end{equation*}
We see from this that $F$ is an isomorphism of abelian groups with an action of $\recht{Gal}(L/K)$. 
Let $t \in T(L)/U$ be Galois invariant, and let $\eta = F^{-1}(t) \in \recht{X}_*(T)^{\recht{Gal}(L/K)}$; 
then $\eta$ is a cocharacter that is defined over $K$. 
By definition we have $\omega^f \in K$, hence $F(f \cdot \eta) = \eta(\omega^f)$ is an element of $T(K)$. 
This shows that the abelian group $\recht{X}_*(T)^{\recht{Gal}(L/K)}/F^{-1}(T(K) \cdot U)$ is annihilated by $f$. 
Since it is finitely generated, it is finite. 
Furthermore, the map $F$ induces a bijection
\begin{equation*}
\recht{X}_*(T)^{\recht{Gal}(L/K)}/F^{-1}(T(K)) \DistTo Q^{L/K}_T(U),
\end{equation*}
hence $Q^{L/K}_G(U)$ is finite.
\end{proof}

\begin{lem} \label{lem:qss}
Let $G$ be a (connected) semisimple group over a $p$-adic field $K$, and let $L$ be a finite Galois extension over which $G$ splits. 
Let $U$ be a Galois invariant compact open subgroup of $G(L)$. 
Then $Q_{G}^{L/K}(U)$ is finite.
\end{lem}

\begin{proof}
By Lemma \ref{lem:oneu} it suffices to show this for a chosen $U$. 
Let $\mathcal{I}(G,K)$ be the Bruhat--Tits building of $G$ over $K$, and let $\mathcal{I}(G,L)$ be the Bruhat--Tits building of $G$ over $L$. 
Choose a point $x \in |\mathcal{I}(G,K)| \subset |\mathcal{I}(G,L)|^{\recht{Gal}(L/K)}$; 
its stabiliser $U \subset G(L)$ is a Galois invariant compact open subgroup of $G(L)$ by Theorems \ref{thm:buildings1}.4 and \ref{thm:buildings2}.2. 
Now we can identify $Q^{L/K}_G(U)$ with 
\begin{equation*}
G(K) \backslash (G(L) \cdot x)^{\recht{Gal}(L/K)},
\end{equation*}
so it suffices to show that this set is finite. 
Let $y \in (G(L) \cdot x)^{\recht{Gal}(L/K)}$, and let $r$ be as in Theorem \ref{thm:buildings2}.4. 
Then there exists a $z \in \mathcal{I}(G,K)$ such that $d(y,z) \leq r$. 
Now fix a chamber $C$ of $\mathcal{I}(G,K)$, and let $g \in G(K)$ such that $gz \in \bar{C}$ (see Theorem \ref{thm:buildings1}.5). 
Then $d(gy,\bar{C}) \leq r$, so $gy$ lies in the set $D = \{v \in |\msc{I}(G,L)| : d(v,\bar{C}) \leq r\}$, which is compact by Corollary \ref{cor:btcompact}. 
On the other hand the action of $G(L)$ on $|\mathcal{I}(G,L)|$ has discrete orbits by Remark \ref{rem:btdiscrete}, 
so $G(L)\cdot x$ intersects $D$ in only finitely many points. 
Hence there are only finitely many possibilities for $gy$, so $G(K) \backslash (G(L) \cdot x)^{\recht{Gal}(L/K)}$ is finite, as was to be shown.
\end{proof}

\begin{prop} \label{prop:qred}
Let $G$ be a connected reductive group over a $p$-adic field $K$, and let $L$ be a finite Galois extension of $K$ over which $G$ splits. 
Let $U$ be a Galois invariant subgroup of $G(L)$. 
Then $Q_G^{L/K}(U)$ is finite.
\end{prop}

\begin{proof}
Let $G'$ be the semisimple group $G^{\recht{der}}$, and let $G^{\recht{ab}}$ be the torus $G/G'$. 
This gives us an exact sequence
\begin{equation*}
1 \rightarrow G'(K) \rightarrow G(K) \stackrel{\psi}{\rightarrow} G^\recht{ab}(K) \rightarrow \recht{H}^1(G',K).
\end{equation*}
The image $\psi(U) \subset G^{\recht{ab}}(L)$ is compact. 
It is also open: 
If $Z$ is the centre of $G$, then the map $\psi\colon Z \rightarrow G^{\recht{ab}}$ is an isogeny, and since $Z(L) \cap U$ is open in $Z(L)$, its image in $G^{\recht{ab}}$ is open as well. 
Hence by Lemma \ref{lem:qtorus} we know that $Q_{G^\recht{ab}}^{L/K}(\psi(U))$ is finite. 
Furthermore, by \cite[{}III.4.3]{serre1997} $\recht{H}^1(G',K)$ is finite, hence the image of $G(K)$ in $G^{\recht{ab}}(K)$ has finite index. 
If we let $G(K)$ act on $(G^{\recht{ab}}(L)/\psi(U))^{\recht{Gal}(L/K)}$ via $\psi$, we now find that the quotient set $G(K) \backslash (G^{\recht{ab}}(L)/\psi(U))^{\recht{Gal}(L/K)}$ is finite. 
The projection map $\psi\colon (G(L)/U)^{\recht{Gal}(L/K)} \rightarrow (G^{\recht{ab}}(L)/\psi(U))^{\recht{Gal}(L/K)}$ is $G(K)$-equivariant, so we get a map of $G(K)$-quotients $\bar{\psi}\colon Q_G^{L/K}(U) \rightarrow G(K)\backslash(G^{\recht{ab}}(L)/\psi(U))^{\recht{Gal}(L/K)}$. 
To show that $Q_G^{L/K}(U)$ is finite it now suffices to show that the map $\bar{\psi}$ is finite.

Let $x \in Q_G^{L/K}(U)$; we need to prove that there exist only finitely many $y \in Q_G^{L/K}(U)$ such that $\bar{\psi}(x) = \bar{\psi}(y)$.
Choose a representative $\tilde{x}$ of $x$ in $G(L)$; then there exists a representative $\tilde{y}$ of $y$ in $G(L)$ such that $\tilde{x} = \tilde{y}$ in $G^{\recht{ab}}(L)$.
Hence there is a $g' \in G'(L)$ such that $g'\tilde{x} = \tilde{y}$. 
Since $\tilde{x}U$ and $\tilde{y}U$ are Galois invariant, the element $g'$ is Galois invariant in $G'(L)/(G'(L) \cap \tilde{x}U\tilde{x}^{-1})$; 
this makes sense because the compact open subgroup $G'(L) \cap \tilde{x}U\tilde{x}^{-1}$ of $G'(L)$ is Galois-invariant. 
Furthermore the element $y$ only depends on the choice of $g'$ in
\begin{equation*}
G'(K) \backslash (G'(L)/(G'(L) \cap \tilde{x}U\tilde{x}^{-1}))^{\recht{Gal}(L/K)} = Q_{G'}^{L/K}(G'(L) \cap \tilde{x}U\tilde{x}^{-1}).
\end{equation*}
Since this set is finite by Lemma \ref{lem:qss} there are only finitely many possibilities for $y$ for a given $x$. 
This proves the Proposition.
\end{proof}

The final Proposition of this section is a stronger version of Proposition \ref{prop:qred} in the case that the compact open subgroup $U$ comes from a `nice' model of $G$. 
We need this to prove a stronger version of Theorem \ref{thm:main1} over local fields in the case that we have models over a collection of local fields coming from the places of some number field (compare Theorem \ref{thm:split}.2).

\begin{prop} \label{smthfin}
Let $K$ be a $p$-adic field, and let $\msc{G}$ be a smooth group scheme over $\msc{O}_K$ whose generic fibre is reductive and splits over an unramified Galois extension $L/K$. 
Then $\#Q_{\msc{G}_K}^{L/K}(\msc{G}(\msc{O}_L)) = 1$.
\end{prop}

\begin{proof}
Let $k$ be the residue field of $K$. 
Let $g \in \msc{G}(L)$ such that $g\msc{G}(\msc{O}_L)$ is Galois-invariant; 
we need to show that  $g\msc{G}(\msc{O}_L)$ has a point defined over $K$. 
Since $L/K$ is unramified, we see that $\recht{Gal}(L/K)$ is the étale fundamental group of the covering $\recht{Spec}(\msc{O}_L)/\recht{Spec}(\msc{O}_K)$. 
As such $g\msc{G}(\msc{O}_L)$ can be seen as the $\msc{O}_L$-points of a $\msc{G}$-torsor $\msc{B}$ over $\recht{Spec}(\msc{O}_K)$ in the sense of \cite[{}III.4]{milne2016}. 
By Lang's theorem the $\msc{G}_k$-torsor $\msc{B}_k$ is trivial, hence $\msc{B}(k)$ is nonempty. 
Since $\msc{G}$ is smooth over $\msc{O}_K$, so is $\msc{B}$, and we can lift a point of $\msc{B}(k)$ to a point of $\msc{B}(\msc{O}_K)$. 
Hence $g\msc{G}(\msc{O}_L)$ has an $\msc{O}_K$-point, as was to be shown.
\end{proof}

\subsection{Models of reductive groups}

In this subsection we prove Theorem \ref{thm:main1} over local fields, plus a stronger statement for local fields coming from one number field.
We need this to prove Theorem \ref{thm:main1} for number fields.

\begin{thm} \label{thm:local}
Let $G$ be a connected reductive group over $K$. 
Let $V$ be a faithful representation of $G$, and regard $G$ as an algebraic subgroup of $\mathtt{GL}(V)$. 
Let $N$ be the scheme-theoretic normaliser of $G$ in $\mathtt{GL}(V)$.
\begin{enumerate}
\item Let $K$ be a $p$-adic field. 
Then the map $\mathtt{mdl}_G\colon \recht{Lat}_N(V) \rightarrow \recht{Mdl}(G)$ of Lemma \ref{lem:mdldef} is finite.
\item Let $K$ be a number field. 
Then there exists a finite Galois extension $L$ of $K$ over which $G$ splits with the following property: 
For almost all finite places $v$ of $K$ there is exactly one $N(K_v)$-orbit $X_v$ of lattices in $V_{K_v}$ such that $\mathtt{mdl}_{G_{K_v}}(X_v)_{\msc{O}_{L_w}}$ is the Chevalley model of $G_{L_w}$ for all places $w$ of $L$ over $v$.
\end{enumerate}
\end{thm}

\begin{proof} Let $N^0$ and $\pi_0(N)$ be the identity component and component group of $N$, respectively.
\begin{enumerate}
\item Let $L/K$ be a Galois extension over which $G$ splits. 
Let $R$ and $S$ be the rings of integers of $K$ and $L$, respectively.
Then we have the following commutative diagram:

\begin{center}
\begin{tikzpicture}[description/.style={fill=white,inner sep=2pt}]
\matrix (m) [matrix of math nodes, row sep=3em,
column sep=6em, text height=1.5ex, text depth=0.25ex]
{ \recht{Lat}_{N^0}(V) & \recht{Lat}_{N^0}(V_L) \\
\recht{Lat}_{N}(V) & \recht{Lat}_{N}(V_L) \\
\recht{Mdl}(G) & \recht{Mdl}(G_L) \\ };
\path[->>] 
(m-1-1) edge (m-2-1)
(m-1-2) edge (m-2-2);
\path[->,font=\scriptsize]
(m-1-1) edge node[auto] {$ S \otimes_{R} - $} (m-1-2)
(m-2-1) edge node[auto] {$ S \otimes_{R} - $} (m-2-2)
edge node[auto]{$\mathtt{mdl}_{G}$}(m-3-1)
(m-2-2) edge node[auto]{$\mathtt{mdl}_{G_L}$} (m-3-2)
(m-3-1) edge node[auto] {$\recht{Spec} S \times_{\recht{Spec} R} -$} (m-3-2);
\end{tikzpicture}
\end{center}
By Theorem \ref{thm:split}.1 we know that the map on the lower right is finite. 
Furthermore, since $N^0$ is of finite index in $N$, we know that the maps on the upper left and upper right are finite and surjective. 
To show that the map on the lower left is finite, it now suffices to show that the top map is finite. 
Let $\Lambda$ be a lattice in $V$. 
The $N^0(L)$-orbit of $\Lambda_S$ in $\recht{Lat}(V_L)$ is a Galois-invariant element of $\recht{Lat}_{N^0}(V_L)$. 
As a set with an $N^0(L)$-action and a Galois action, this set is isomorphic to $N^0(L)/U$, where $U \subset N^0(L)$ is the stabiliser of $\Lambda_S$; 
this is a compact open Galois-invariant subgroup of $N^0(L)$. 
If $\Lambda' \in \recht{Lat}(V)$ is another lattice such that $\Lambda'_S \in N^0(L) \cdot \Lambda_S$, then $\Lambda'_S$ corresponds to a Galois-invariant element of $N^0(L)/U$. 
By \cite{user27056} we know that $N^0$ is reductive, hence $Q_{N^0}^{L/K}(U)$ is finite by Proposition \ref{prop:qred}. 
This shows that, given $\Lambda$, there are only finitely many options for $N^0(K) \cdot \Lambda'$. 
Hence the top map of the above diagram is finite, as was to be shown. 
 
 \item Let $L/K$ be finite Galois such that the map $N(L) \rightarrow \pi_0(N)(\bar{K})$ is surjective. 
 Choose a lattice $\Lambda \in \recht{Lat}(V)$. 
Let $\msc{N}^0$ be the model of $N^0$ induced by $\Lambda$. 
Let $\msc{N}^0_v := \msc{N}^0_{R_v}$; this is the model of $N^0_{K_v}$ induced by $\Lambda_{R_v} \subset V_{K_v}$. 
For almost all $v$ the $R_v$-group scheme $\msc{N}^0_v$ is reductive. 
Since $G_L$ is split, for almost all places $w$ of $L$ the model of $G_{L_w}$ associated to $\Lambda_{S_w}$ is the Chevalley model.
Furthermore, let $n_1,\ldots,n_r \in N(L)$ be a set of representatives of $\pi_0(N)(\bar{K})$; then for every place $w$ of $L$ we have
 \begin{equation*}
 N(L_w) \cdot \Lambda_{S_w} = \bigcup_{i=1}^r N^0(L_w)n_i \cdot \Lambda_{S_w}.
 \end{equation*}
 For almost all $w$ all the lattices $n_i \cdot \Lambda_{S_w}$ coincide, hence for those $w$ we have $N(L_w) \cdot \Lambda_{S_w} = N^0(L_w) \cdot \Lambda_{S_w}$. 
 Now let $v$ be a finite place of $K$ satisfying the following conditions:
 \begin{itemize}
 \item For every place $w$ of $L$ above $v$, the $N(L_w)$-orbit of lattices $N(L_w) \cdot \Lambda_{S_w}$ is the only orbit of lattices in $V_{L_w}$ inducing the Chevalley model of $G_{L_w}$;
 \item  for every place $w$ of $L$ above $v$ we have  $N(L_w) \cdot \Lambda_{S_w} = N^0(L_w) \cdot \Lambda_{S_w}$;
 \item $L$ is unramified over $v$;
 \item $\msc{N}^0_v$ is reductive.
 \end{itemize}
The last three conditions hold for almost all $v$, and by Theorem \ref{thm:split}.2 the same is true for the first condition. 
Let us now follow the proof of the previous point, for the group $G_{K_v}$ and its faithful representation $V_{K_v}$. 
The first two conditions tell us that $N^0(L_w) \cdot \Lambda_{S_w}$ is the only $N^0(L_w)$-orbit of lattices yielding the Chevalley model of $G_{L_w}$ for every place $w$ of $L$ over $v$. 
By the last two conditions and Proposition \ref{smthfin} we know that $Q_{N^0}^{L_w/K_v}(\msc{N}^0(S_w)) = 1$, hence there is only one $N^0(K_v)$-orbit of lattices that gets mapped to $N^0(L_w) \cdot \Lambda_{S_w}$. 
This is the unique $N^0(K_v)$-orbit of lattices in $V_{K_v}$ yielding the Chevalley model of $G_{L_w}$; 
in particular there is only one $N(K_v)$-orbit of such lattices. \qedhere
\end{enumerate}
\end{proof}

\section{Reductive groups over number fields} \label{s:numberfield}

In this section we prove Theorem \ref{thm:main1} over number fields. 
To go from local fields to number fields, we work with the topological ring of finite adèles $\mathbb{A}_{K,\recht{f}}$ over a number field $K$; 
let $\hat{R} \subset \mathbb{A}_{K,\recht{f}}$ be the profinite completion of the ring of integers $R$ of $K$. 
If $M$ is a free $\mathbb{A}_{K,\recht{f}}$-module of finite rank, we say that a \emph{lattice} in $M$ is a free $\hat{R}$-submodule that generates $M$ as an $\mathbb{A}_{K,\recht{f}}$-module. 
The set of lattices in $M$ is denoted $\recht{Lat}(M)$, and if $G$ is a subgroup scheme of $\mathtt{GL}(M)$, we denote $\recht{Lat}_G(M) := G(\mathbb{A}_{K,\recht{f}}) \backslash \recht{Lat}(M)$.
If $V$ is a finite dimensional $K$-vector space, then the map $\Lambda \mapsto \Lambda_{\hat{R}}$ gives a bijection $\recht{Lat}(V) \DistTo \recht{Lat}(V_{\mathbb{A}_{K,\recht{f}}})$.

\begin{lem} \label{lem:adele}
Let $K$ be a number field, let $G$ be a (not necessarily connected) reductive group over $K$, and let $V$ be a finite dimensional faithful representation of $G$. Let $\msc{G}$ be a model of $G$.
\begin{enumerate}
\item $\msc{G}(\hat{R})$ is a compact open subgroup of $G(\mathbb{A}_{K,\recht{f}})$ in the adèlic topology;
\item The map $\recht{Lat}_G(V) \rightarrow \recht{Lat}_G(V_{\mathbb{A}_{K,\recht{f}}}) $ is finite;
\item The map $\recht{Lat}_G(V_{\mathbb{A}_{K,\recht{f}}}) \rightarrow \prod_v \recht{Lat}_G(V_{K_v})$ is injective.
\end{enumerate}
\end{lem}

\begin{proof} { \ }
\begin{enumerate}
\item Let $V$ be a faithful representation of $G$ and let $\Lambda$ be a lattice in $V$ such that $\msc{G}$ is the model of $G$ associated to $\Lambda$. 
Then $\msc{G}(\hat{R}) = G(\mathbb{A}_{K,\recht{f}}) \cap \recht{End}(\Lambda_{\hat{R}})$. 
Since $\recht{End}(\Lambda_{\hat{R}})$ is open in $\recht{End}(V_{\mathbb{A}_{K,\recht{f}}})$, we see that $\msc{G}(\hat{R})$ is open in $G(\mathbb{A}_{K,\recht{f}})$. 
It is compact because it is the profinite limit of finite groups $\varprojlim \msc{G}(R/I)$, where $I$ ranges over the ideals of $R$.
\item  Let $\Lambda$ be a lattice in $V$, and let $\msc{G}$ be the model of $G$ induced by $\Lambda$. 
Then the stabiliser of $\Lambda_{\hat{R}}$ in $G(\mathbb{A}_{K,\recht{f}})$ is equal to $\msc{G}(\hat{R})$, which by the previous point is a compact open subgroup of $G(\mathbb{A}_{K,\recht{f}})$. 
Then as a $G(\mathbb{A}_{K,\recht{f}})$-set we can identify $G(\mathbb{A}_{K,\recht{f}}) \cdot \Lambda_{\hat{R}}$ with $G(\mathbb{A}_{K,\recht{f}})/\msc{G}(\hat{R})$. 
By \cite[Thm. 5.1]{borel1963} the double coset $G(K) \backslash G(\mathbb{A}_{K,\recht{f}})/\msc{G}(\hat{R})$ is finite. 
This means that $G(\mathbb{A}_{K,\recht{f}}) \cdot \Lambda_{\hat{R}}$ consists of only finitely many $G(K)$-orbits of lattices in $V_{\mathbb{A}_{K,\recht{f}}}$. 
Since the map $\recht{Lat}(V) \rightarrow \recht{Lat}(V_{\mathbb{A}_{K,\recht{f}}})$ is a $G(K)$-equivariant bijection, each of these orbits corresponds to one $G(K)$-orbit of lattices in $V$; 
hence there are only finitely many $G(K)$-orbits of lattices in $V$ with the same image as $\Lambda$ in $\recht{Lat}_G(V_{\mathbb{A}_{K,\recht{f}}})$, which proves that the given map is indeed finite.
\item Let $\Lambda, \Lambda'$ be two lattices in $V_{\mathbb{A}_{K,\recht{f}}}$ whose images in $\prod_v \recht{Lat}_G(V_{K_v})$ are the same. 
Then for every $v$ there exists a $g_v \in G(K_v)$ such that $g_v \cdot \Lambda_{R_v} = \Lambda'_{R_v}$. 
Since $\Lambda_{R_v} = \Lambda'_{R_v}$ for almost all $v$, we can take $g_v = 1$ for almost all $v$;
 hence $g \cdot \Lambda = \Lambda'$ for $g = (g_v)_v \in G(\mathbb{A}_{K,\recht{f}})$. \qedhere
\end{enumerate}
\end{proof}

\begin{proof}[Proof of Theorem \ref{thm:main1}] 
The case that $K$ is a $p$-adic field is proven in Theorem \ref{thm:local}.1, so suppose $K$ is a number field. 
Then we have the following commutative diagram:

\begin{center}
\begin{tikzpicture}[description/.style={fill=white,inner sep=2pt}]
\matrix (m) [matrix of math nodes, row sep=3em,
column sep=6em, text height=1.5ex, text depth=0.25ex]
{\recht{Lat}_N(V) & \recht{Lat}_N(V_{\mathbb{A}_{K,\recht{f}}}) & \prod_v \recht{Lat}_N(V_{K_v}) \\
\recht{Mdl}(G) & & \prod_v \recht{Mdl}(G_{K_v})\\};
\path[->,font=\scriptsize]
(m-1-1) edge node[auto] {$ f_1$}  (m-1-2)
edge node[auto] {$\mathtt{mdl}_G$} (m-2-1)
(m-1-2) edge node[auto] {$ f_2$} (m-1-3)
(m-1-3) edge node[auto] {$\prod_v \mathtt{mdl}_{G_{K_v}}$} (m-2-3)
(m-2-1) edge node[auto] {$\prod_v \recht{Spec}(R_v) \times_{\recht{Spec}(R)} -$} (m-2-3);
\end{tikzpicture}
\end{center}

Let $L$ be as in Theorem \ref{thm:local}.2, and let $R$ and $S$ be the rings of integers of $K$ and $L$, respectively. 
Let $\msc{G}$ be a model of $G$. 
Then for almost all finite places $w$ of $L$ the model $\msc{G}_{S_w}$ of $G_{L_w}$ is the Chevalley model. 
By Theorem \ref{thm:local} we know that for every finite place $v$ of $K$ there are only finitely many $N(K_v)$-orbits of lattices in $V_{K_v}$ whose associated model is $\msc{G}_{R_v}$, and for almost all $v$ there is exactly one such orbit. 
This shows that there are only finitely many elements of $\prod_v \recht{Lat}_N(V_{K_v})$ that map to $(\msc{G}_{R_v})_v$. 
Hence the map on the right of the diagram above is finite; 
since $f_1$ and $f_2$ are finite as well by Lemma \ref{lem:adele}, this proves the Theorem.
\end{proof}

\begin{rem} \label{rem:locmod}
The proof of Theorem \ref{thm:main1} also shows that for every collection of models $(\msc{G}_v)_v$ of the $G_v$, there are at most finitely many $N(K)$-orbits of lattices in $V$ that yield that collection of models.
\end{rem}

\section{Integral Mumford--Tate groups} \label{s:shimura}

Let $g$ and $n$ be integers with $g \geq 1$ and $n \geq 3$.
Let $\mathcal{A}_{g,n}$ be the moduli space of principally polarised abelian varieties of dimension $g$ with level $n$ structure. 
Let $\mathcal{X}_{g,n}$ be the universal abelian variety over $\mathcal{A}_{g,n}$, and let $\mathcal{V}_{g,n}$ be the variation of integral Hodge structures on $\mathcal{A}_{g,n}$ for which $\mathcal{V}_{g,n,y} = \recht{H}^1(\mathcal{X}^{\recht{an}}_{g,n,y},\mathbb{Z})$ for every $y \in \mathcal{A}_{g,n}(\mathbb{C})$.
If $Y \subset \mathcal{A}_{g,n}$ is a special subvariety, then we can define its \emph{generic (integral) Mumford--Tate group} $\mathtt{GMT}(Y)$ analogously to how one defines the generic rational Mumford--Tate group for a variation of rational Hodge structures, as for instance in \cite{moonen2004introduction}. 
If $y \in Y$ is Hodge generic (i.e. its special closure is $Y$) then $\mathtt{GMT}(Y)$ is isomorphic to the Mumford--Tate group of the integral Hodge structure $\mathcal{V}_{g,n,y}$.
The resulting integral group scheme is flat and of finite type over $\mathbb{Z}$, and its generic fibre is a connected reductive rational algebraic group. The aim of this section is to prove the following Theorem:

\begin{thm} \label{thm:main2}
Let $g$ and $n$ be positive integers with $n > 2$, and let $\mathscr{G}$ be a group scheme over~$\mathbb{Z}$. 
Then there are at most finitely many special subvarieties $Y$ of $\mathcal{A}_{g,n}$ such that $\mathtt{GMT}(Y) \cong \mathscr{G}$.
\end{thm}

Throughout this section, by a \emph{symplectic representation} of an algebraic group $G$ over a field $K$ we mean a morphism of algebraic groups $G \rightarrow \mathtt{GSp}(V,\psi)$ for some symplectic $K$-vector space $(V,\psi)$. 
By \cite[Thm.~2.1(b)]{knop2006} the isomorphism class of a symplectic representation is uniquely determined by its underlying representation $G \rightarrow \mathtt{GL}(V)$. 
Before we prove the main result in section \ref{ss:integralcase}, we first need to devote section \ref{ss:rationalcase} to Proposition \ref{prop:main2rat}, which is the `rational analogon' to Theorem \ref{thm:main2}.

\subsection{The rational case} \label{ss:rationalcase}

The goal of this section is to prove the following Proposition:

\begin{prop} \label{prop:main2rat}
Let $G$ be a reductive algebraic group over $\mathbb{Q}$.
then up to Hecke correspondence there are at most finitely many special subvarieties $Y$ of $\mathcal{A}_{g,n}$ such that $\mathtt{GMT}(Y)_{\mathbb{Q}} \cong G$.
\end{prop}

This is as far as we can get into proving Theorem \ref{thm:main2} without using integral information, as rational generic Mumford--Tate groups are invariant under Hecke correspondences.
We first need to set up some notation before we get to the proof.
For an algebraic group $G$ over $\mathbb{Q}$ we write $G(\mathbb{R})^+$ for the identity component of the Lie group $G(\mathbb{R})$.
We write $\mathbb{S}$ for the Deligne torus $\recht{Res}_{\mathbb{C}/\mathbb{R}}\mathbb{G}_{\recht{m}}$.
Let $\mathcal{H}_g$ be the $g$-dimensional Siegel space; then $(\mathtt{GSp}_{2g},\mathcal{H}_g)$ is a Shimura datum. 
We fix a connected component $\mathcal{H}_g^+ \subset \mathcal{H}_g$.

\begin{defn}
A \emph{reductive connected Shimura datum} is a pair $(G,X^+)$ consisting of a connected reductive group $G$ over $\mathbb{Q}$ and a $G(\mathbb{R})^+$-orbit $X^+$ of morphisms $\mathbb{S} \rightarrow G_{\mathbb{R}}$ such that the pair $(G,G(\mathbb{R})\cdot X^+)$ is a Shimura datum.
\end{defn}

A reductive connected Shimura datum differs from a connected Shimura datum in the sense of \cite{milne2004introduction} in that we do not require $G$ to be semisimple, and we look at morphisms $\mathbb{S} \rightarrow G_{\mathbb{R}}$ rather than morphisms $\mathbb{S}^1 \rightarrow G^{\recht{ad}}_{\mathbb{R}}$.

\begin{defn}
A \emph{connected Shimura triple of rank $2g$} is a triple $(G,X^+,\varrho)$ consisting of:
\begin{itemize}
\item a reductive algebraic group $G$ over $\mathbb{Q}$;
\item a $G(\mathbb{R})^+$-orbit $X^+ \subset \recht{Hom}(\mathbb{S},G_{\mathbb{R}})$ such that the pair $(G,X^+)$ is a connected reductive Shimura datum;
\item an injective morphism of algebraic groups $\varrho\colon G \hookrightarrow \mathtt{GSp}_{2g,\mathbb{Q}}$ such that $\varrho_{\mathbb{R}} \circ X^+ \subset \mathcal{H}^+_g$, and such that $G$ is the generic Mumford--Tate group of $X^+$ under this embedding.
\end{itemize}
\end{defn}

For a reductive group $G$, we define 
\begin{equation}
\Omega(G,g) := \Big\{(X^+,\varrho): (G,X^+,\varrho) \textrm{ is a connected Shimura triple of rank $2g$}\Big\}. 
\end{equation}
There are commuting actions from $\recht{Aut}(G)$ on the left and $\recht{GSp}_{2g}(\mathbb{Q})$ on the right on $\Omega(G,g)$, defined by
\begin{equation} \label{eq:omegact}
\sigma \cdot (X^+,\varrho) \cdot a := (\sigma_{\mathbb{R}} \circ X^+,\recht{inn}(a)^{-1} \circ \varrho \circ \sigma^{-1}).
\end{equation}
The reason to study these Shimura triples is that every special subvariety comes from a connected Shimura triple in the following sense:
The Shimura variety $\mathcal{A}_{g,n}$ is a finite disjoint union of complex analytical spaces of the form $\Gamma \backslash \mathcal{H}^+_g$, where $\Gamma \subset \recht{GSp}_{2g}(\mathbb{Z})$ is a congruence subgroup, and $\mathcal{H}^+_g$ is a fixed connected component of $\mathcal{H}_g$. 
For such a $\Gamma$, and a connected Shimura triple $(G,X^+,\varrho)$ of rank $2g$, denote by $Y_{\Gamma}(G,X^+,\varrho)$ the image of $\varrho(X^+) \subset \mathcal{H}^+_g$ in $\Gamma \backslash \mathcal{H}_g^+$.
This is a special subvariety of $\Gamma \backslash \mathcal{H}_g^+$, and all special subvarieties arise in this way.
Furthermore, $\mathtt{GMT}(Y_{\Gamma}(G,X^+,\varrho))_{\mathbb{Q}}$ is isomorphic to $G$. If $Y = Y_{\Gamma}(G,X^+,\varrho)$ and $Y'$ are two special subvarieties of $\Gamma  \backslash \mathcal{H}^+_g$ that differ by a Hecke correspondence, then there exists an $a \in \recht{GSp}_{2g}(\mathbb{Q})$ such that $Y' = Y_{\Gamma}(G,X^+,\recht{inn}(a)^{-1} \circ \varrho)$.
Proposition \ref{prop:main2rat} is therefore equivalent to the following result:

\begin{prop}
Let $G$ be a connected reductive group over $\mathbb{Q}$. 
Then the cardinality of the quotient set $\recht{Aut}(G)\backslash\Omega(G,g)/\recht{GSp}_{2g}(\mathbb{Q})$ is finite.
\end{prop}

The rest of this subsection is dedicated to the proof of this Proposition.
We first prove some auxiliary results.

\begin{lem} \label{lem:detfin}
Let $d$ be a positive integer.
Let $\Pi$ be a finite subgroup of $\recht{GL}_d(\mathbb{Z})$, and let $\eta_0 \in \mathbb{Z}^d$ be such that $\Pi \cdot \eta_0$ generates the rational vector space $\mathbb{Q}^d$.
Then up to the action of $\recht{Aut}_{\Pi}(\mathbb{Z}^d)$ there are only finitely many elements $\eta \in \mathbb{Z}^d$ such that for all $\pi_1,\cdots,\pi_d \in \Pi$ we have
\begin{equation} \label{eq:piaction}
\recht{det}(\pi_1 \cdot \eta_0,\cdots,\pi_d \cdot \eta_0) = \recht{det}(\pi_1 \cdot \eta,\cdots,\pi_d \cdot \eta).
\end{equation}
\end{lem}

\begin{proof}
Fix $\pi_1,\cdots,\pi_d \in \Pi$ such that the $\pi_i \cdot \eta_0$ are $\mathbb{Q}$-linearly independent,
and define the integer $C := \recht{det}(\pi_1 \cdot \eta_0,\cdots,\pi_d \cdot \eta_0)$.
Let $\eta \in \mathbb{Z}^d$ be such that it satisfies (\ref{eq:piaction}).
Then $ \recht{det}(\pi_1 \cdot \eta,\cdots,\pi_d \cdot \eta) = C \neq 0$, so the $\pi_i \cdot \eta$ are $\mathbb{Q}$-linearly independent as well.
If $\pi$ is any element of $\Pi$, then there exist $c_i,c_i' \in \mathbb{Q}$ such that $\pi \cdot \eta_0 = \sum_i c_i(\pi_i \cdot \eta_0)$ and $\pi \cdot \eta = \sum_i c'_i(\pi_i \cdot \eta)$. 
Then we may calculate
\begin{align*}
c_i \cdot C &= \recht{det}(\pi_1 \cdot \eta_0,\cdots, \pi_{i-1}\cdot \eta_0, \pi \cdot \eta_0,\pi_{i+1} \cdot \eta_0,\cdots,\pi_d \cdot \eta_0) \\
&= \recht{det}(\pi_1 \cdot \eta,\cdots, \pi_{i-1}\cdot \eta, \pi \cdot \eta,\pi_{i+1} \cdot \eta,\cdots,\pi_d \cdot \eta) \\
&= c'_i \cdot C.
\end{align*}
Hence $c_i = c'_i$, and we find that for every collection of schalars $(x_\pi)_{\pi \in \Pi}$ in $\recht{Map}(\Pi,\mathbb{Q})$ we have
\begin{equation*}
\sum_{\pi \in \Pi} x_{\pi} \cdot (\pi \cdot \eta_0) = 0 \Leftrightarrow \sum_{\pi \in \Pi} x_{\pi} \cdot (\pi \cdot \eta) = 0.
\end{equation*}
It follows that there exists a unique $\Pi$-equivariant linear isomorphism $f_{\eta}\colon \mathbb{Q}^d \rightarrow \mathbb{Q}^d$ satisfying $f_{\eta}(\eta_0) = \eta$.
Let $\Lambda_{\eta}$ be the lattice in $\mathbb{Q}^d$ generated by $\Pi \cdot \eta$;
then $f_{\eta}(\Lambda_{\eta_0}) = \Lambda_\eta$. Now let $\eta' \in \mathbb{Z}^d$ be another element satisfying (\ref{eq:piaction});
then $f_{\eta'} \circ f_{\eta}^{-1}$ is the unique $\Pi$-equivariant automorphism of $\mathbb{Q}^d$ that sends $\eta$ to $\eta'$.
This automorphism induces a $\Pi$-equivariant automorphism of $\mathbb{Z}^d$ if and only if $f^{-1}_{\eta}(\mathbb{Z}^d) = f^{-1}_{\eta'}(\mathbb{Z}^d)$ in $\mathbb{Q}^d$.
Hence $\recht{Aut}_{\Pi}(\mathbb{Z}^d)$-orbits of suitable $\eta$ correspond bijectively to lattices of the form $f^{-1}_{\eta}(\mathbb{Z}^d)$ in $\mathbb{Q}^d$.
Let $C$ be as above;
then $\Lambda_{\eta} \subset \mathbb{Z}^d \subset C^{-1}\Lambda_{\eta}$, hence $\Lambda_{\eta_0} \subset f^{-1}_{\eta}(\mathbb{Z}^d) \subset C^{-1} \Lambda_{\eta_0}$.
Since there are only finitely many options for lattices between $\Lambda_{\eta_0}$ and $C^{-1}\Lambda_{\eta_0}$, we conclude that there are only finitely many options for the $\recht{Aut}_{\Pi}(\mathbb{Z}^d)$-orbit of $\eta$.
\end{proof}

\begin{lem} \label{lem:torfin}
Let $T$ be a torus over $\mathbb{Q}$, and let $\nu\colon \mathbb{G}_{\recht{m},\bar{\mathbb{Q}}} \rightarrow \mathtt{GSp}_{2g,\bar{\mathbb{Q}}}$ be a symplectic representation.
Let $S$ be the collection of pairs $(\eta,\tau)$, where $\eta\colon \mathbb{G}_{\recht{m},\bar{\mathbb{Q}}} \rightarrow T_{\bar{\mathbb{Q}}}$ is a cocharacter whose image is Zariski dense in the $\mathbb{Q}$-group $T$, and $\tau\colon T \hookrightarrow \mathtt{GSp}_{2g,\mathbb{Q}}$ is a faithful symplectic representation, such that $\nu \cong \tau_{\bar{\mathbb{Q}}} \circ \eta$ as symplectic representations of $\mathbb{G}_{\recht{m},\bar{\mathbb{Q}}}$. Define an action of $\recht{Aut}(T)$ on $S$ by $\sigma \cdot (\eta,\tau) = (\sigma_{\bar{\mathbb{Q}}} \circ \eta, \tau \circ \sigma^{-1})$. Then $\recht{Aut}(T) \backslash S$ is finite.
\end{lem}

\begin{proof}
Define $X = \recht{X}_*(T)$ as a free abelian group with a $\recht{Gal}(\bar{\mathbb{Q}}/\mathbb{Q})$-action, and identify $\recht{X}^*(T)$ with $X^{\vee}$ via the natural perfect pairing.
Let $\Pi$ be the image of $\recht{Gal}(\bar{\mathbb{Q}}/\mathbb{Q})$ in $\recht{GL}(X)$; this is a finite group.
Now let $(\eta,\tau) \in S$; then $\tau$ is given by a multiset $W \subset X^{\vee}$.
The fact that $\tau$ is faithful and defined over $\mathbb{Q}$ implies that $W$ generates $X^{\vee}$ as an abelian group and that $W$ is invariant under the action of $\Pi$.
Since the image of $\eta$ is Zariski dense in $T$, we find that $X_{\mathbb{Q}}$ is generated by $\Pi \cdot \eta$.
Now let $d$ be the rank of $X$, and let $\pi_1,\cdots,\pi_d \in \Pi$. Consider the homomorphism of abelian groups
\begin{align*}
\varphi_{\eta,(\pi_i)_i}\colon X^{\vee} &\rightarrow \mathbb{Z}^d \\
\lambda & \mapsto (\lambda(\pi_i \cdot \eta))_{i \leq d}.
\end{align*}
The isomorphism class of the representation $\nu$ is given by a multiset $\Sigma \subset \recht{X}^*(\mathbb{G}_{\recht{m}}) = \mathbb{Z}$. 
Since we require $\nu \cong \tau_{\mathbb{Q}} \circ \eta$, we find that $W \circ \eta = \Sigma$ as multisets in $\mathbb{Z}$. 
Furthermore, $W$ is Galois-invariant, so $W \circ (\pi \cdot \eta) = \Sigma$ for all $\pi \in \Pi$. Define
\begin{equation*}
m := \recht{max}\Big\{|\sigma|:\sigma \in \Sigma \subset \mathbb{Z}\Big\}.
\end{equation*}
Then the multiset $\varphi_{\eta,(\pi_i)_i}(W)$ in $\mathbb{Z}^d$ is contained in $\{-m,\cdots,m\}^d$. 
Choose an identification $X \cong \mathbb{Z}^d$, so that we may consider $\varphi_{\eta,(\pi_i)_i}$ as an element of $\recht{Mat}_d(\mathbb{Z})$.
In this notation, the $i$-th column of $\varphi_{\eta,(\pi_i)_i}$ equals the column vector $\pi_i \cdot \eta$.
Furthermore, $|\recht{det}(\varphi_{\eta,(\pi_i)_i})|$ is equal to the volume of the image of a fundamental parallellogram of $\mathbb{Z}^d$. 
Since $X^{\vee}$ is generated by $W$, this volume cannot exceed $m^d$, hence $|\recht{det}(\varphi_{\eta,(\pi_i)_i})| \leq m^d$ for all choices of the $\pi_i$.
Hence, if we let $(\tau,\eta)$ range over $S$, there are only finitely many possibilities for the map
\begin{align*}
t_\eta\colon \Pi^d &\rightarrow \{-m^d,\cdots,m^d\} \\
(\pi_1,\cdots,\pi_d) &\mapsto \recht{det}(\varphi_{\eta,\pi_i}).
\end{align*}
By Lemma \ref{lem:detfin} there are, up to the action of $\recht{Aut}(T) \cong \recht{Aut}_{\Pi}(X)$, only finitely many $\eta \in X$ yielding the same $t_{\eta}$.
Since the set of possible $t_{\eta}$ is also finite, we see that there are only finitely many options for $\eta$ (up to the $\recht{Aut}(T)$-action).
 Now fix such an $\eta$. 
 For every $w \in W$ we need to have $w(\pi \cdot \eta) \in \Sigma$, for all $\pi \in \Pi$. 
 Since $\Pi \cdot \eta$ generates $X_{\mathbb{Q}}$, there are only finitely many options for $w$, hence for the multiset $W$, since the cardinality of $W$ has to be equal to $2g$.
 We conclude that up to the action of $\recht{Aut}(T)$ there are only finitely many possibilities for $(\eta,\tau)$.
\end{proof}

\begin{lem} \label{lem:autgfin}
Let $G$ be a connected reductive group over $\mathbb{Q}$, and let $Z$ be the identity component of its centre.
Let $\varphi$ be the natural map $\recht{Aut}(G) \rightarrow \recht{Aut}(Z)$.
Then $\varphi(\recht{Aut}(G))$ has finite index in $\recht{Aut}(Z)$.
\end{lem}

\begin{proof}
Let $H$ be the finite group $Z \cap G^{\recht{der}}$, and let $n := \#H$. 
If $\sigma$ is an automorphism of $Z$ that is the identity on $H$, then we can extend $\sigma$ to an automorphism $\tilde{\sigma}$ of $G$ by having $\tilde{\sigma}$ be the identity on $G^{\recht{der}}$.
Because of this, it suffices to show that the subgroup
\[
\Big\{\sigma \in \recht{Aut}(Z) : \sigma|_H = \recht{id}_H\Big\} \subset \recht{Aut}(Z)
\]
has finite index. Let $X := \recht{X}_*(Z)$. Let $\sigma \in \recht{Aut}(T)$, and consider $\sigma$ as an element of $\recht{GL}(X)$.
If $\sigma$ maps to the identity in $\recht{Aut}_{\mathbb{Z}/n\mathbb{Z}}(X/nX)$, then $\sigma$ is the identity on $Z[n]$, and in particular on $H$.
Since $\recht{Aut}_{\mathbb{Z}/n\mathbb{Z}}(X/nX)$ is finite, the Lemma follows.
\end{proof}

\begin{lem} \label{lem:repdercent}
Let $G$ be a connected reductive group over $\mathbb{Q}$, and let $Z$ be the identity component of its centre.
Let $\varrho_{\recht{cent}}$ and $\varrho_{\recht{der}}$ be $2g$-dimensional symplectic representations of $Z$ and $G^{\recht{der}}$.
Then there are at most finitely many symplectic representations $\varrho$ of $G$ such that $\varrho|_{Z} \cong \varrho_{\recht{cent}}$ and $\varrho|_{G^{\recht{der}}} \cong \varrho_{\recht{der}}$ as symplectic representations of $Z$ and $G^\recht{der}$, respectively.
\end{lem}

\begin{proof}
Let $T'$ be a maximal torus of $G^{\recht{der}}$.
Then the isomorphism classes of $\varrho_{\recht{cent}}$ and $\varrho_{\recht{der}}$ are given by multisets $\Sigma_{\recht{cent}} \subset \recht{X}^*(Z)$ and $\Sigma_{\recht{der}} \subset \recht{X}^*(T')$, both of cardinality $2g$.
Let $T := Z \cdot T' \subset G$, this is a maximal torus.
A symplectic representation $\varrho$ of $G$ satisfying these conditions corresponds to a multiset $\Sigma \subset \recht{X}^*(T)$ of cardinality $2g$, such that $\Sigma$ maps to $\Sigma_{\recht{cent}}$ in $\recht{X}^*(Z)$ and to $\Sigma_{\recht{der}}$ in $\recht{X}^*(T')$. 
Because $\recht{X}^*(T)_{\mathbb{Q}} = \recht{X}^*(Z)_{\mathbb{Q}} \oplus \recht{X}^*(T')_{\mathbb{Q}}$, there are only finitely many options for $\Sigma$, as we can obtain all of them by creating pairings of elements of $\Sigma_{\recht{cent}}$ with elements of $\Sigma_{\recht{der}}$.
\end{proof}

For the proof of Proposition \ref{prop:main2rat} we furthermore need some more notation. We fix a $G$ and $g$, and write $\bar{\Omega} := \Omega(G,g)/\recht{GSp}_{2g}(\mathbb{Q})$; we need to show that $\recht{Aut}(G)\backslash\bar{\Omega}$ is finite. Consider the natural projection $Z \times G^{\recht{der}} \twoheadrightarrow G$. This is an isogeny, and we let $n$ be its degree.
Let $(X^+,\varrho)$ be an element of $\Omega$.
If $x$ is an element of $X^+$, then  the composite map $\mathbb{S} \stackrel{n}{\rightarrow} \mathbb{S} \stackrel{x}{\rightarrow} G_{\mathbb{R}}$ factors uniquely through $Z_{\mathbb{R}} \times G^{\recht{der}}_{\mathbb{R}}$.
Let $x_{\recht{cent}}$ and $x_{\recht{der}}$ be the associated maps from $\mathbb{S}$ to $Z_{\mathbb{R}}$ and $G^{\recht{der}}_{\mathbb{R}}$, respectively. 
Define $X^+_{\recht{der}} = \{x_{\recht{der}}:x \in X^+\}$, and let $X^+_{\recht{ad}}$ be the image of $X^+$ under $\recht{Ad}:G \twoheadrightarrow G^{\recht{ad}}$. We define the sets $\bar{\Omega}_{\recht{der}}$ and $\bar{\Omega}_{\recht{cent}}$:
\begin{itemize}
\item Let $\bar{\Omega}_{\recht{der}}$ be the set of all pairs $(Y^+,\sigma)$, where $Y^+$ is a $G^{\recht{der}}(\mathbb{R})^+$-orbit in $\recht{Hom}(\mathbb{S},G^{\recht{der}}_{\mathbb{R}})$ such that $\recht{Ad}_{\mathbb{R}} \circ Y^+ = X^+_{\recht{ad}} \circ n$ for some connected Shimura datum $(G^{\recht{ad}},X^+_{\recht{ad}})$, and $\sigma$ is an isomorphism class of symplectic representations of $G^{\recht{der}}$ of dimension $2g$.
\item Let $\mu\colon \mathbb{G}_{\recht{m},\mathbb{C}}\rightarrow \mathbb{S}_{\mathbb{C}}$ be the Hodge cocharacter. Let $\Xi_{\recht{cent}}$ be the set of all pairs $(\eta,\tau)$, where $\eta \in \recht{X}^*(Z^0)$, and $\tau$ is an isomorphism class of symplectic representations of $Z$ of dimension $2g$. Let $\bar{\Omega}_{\recht{cent}} \subset \Xi_{\recht{cent}}$ be the subset of all pairs of the form $(x_{\recht{cent}} \circ \mu,\varrho|_Z)\big)$, where $(X^+,\varrho) \in \bar{\Omega}$, and $x \in X^+$.
\end{itemize}

Note that there is a natural map
\begin{align*}
\varphi\colon \bar{\Omega} &\rightarrow \bar{\Omega}_{\recht{der}} \times \bar{\Omega}_{\recht{cent}} \\
(X^+,\varrho) &\mapsto \big((X^+_{\recht{der}},\varrho|_{G^{\recht{der}}}),(x_{\recht{cent}} \circ \mu,\varrho|_Z)\big),
\end{align*}
where $x$ is any element of $X^+$; since $x_{\recht{cent}}$ is invariant under the action of $G^{\recht{der}}(\mathbb{R})^+$ on $X^+$, the choice does not matter. We can define left actions $\recht{Aut}(G)$ on $\bar{\Omega}_{\recht{der}}$ and $\bar{\Omega}_{\recht{cent}}$ analogously to (\ref{eq:omegact}), and this makes $\varphi$ an $\recht{Aut}(G)$-equivariant map.

\begin{lem} \label{lem:phifin}
The map $\varphi$ is finite.
\end{lem}

\begin{proof}
This follows directly from Lemma \ref{lem:repdercent}
\end{proof}

\begin{lem} \label{lem:derfin}
The set $\bar{\Omega}_{\recht{der}}$ is finite.
\end{lem}

\begin{proof}
Let $X^+_{\recht{ad}}$ be the image of $X^+$ under $\recht{Ad}\colon G \twoheadrightarrow G^{\recht{ad}}$. Then $(G^{\recht{ad}},X^+_{\recht{ad}})$ is a connected Shimura datum in the sense of \cite{milne2004introduction}. It follows from \cite{deligne} that, for a given $G$, there are only finitely many possibilities for $X^+_{\recht{ad}}$. Since $\recht{Ad}_{\mathbb{R}} \circ X^+_{\recht{der}} = X^+_{\recht{ad}} \circ n$ as subsets of $\recht{Hom}(\mathbb{S},G^{\recht{ad}}_{\mathbb{R}})$ and $\recht{Ad}\colon G^{\recht{der}} \twoheadrightarrow G^{\recht{ad}}$ is an isogeny, there are only finitely many possibilities for $Y^+$. Furthermore, a semisimple group has only finitely many symplectic representations of a given dimension, so there are also only finitely many possibilities for $\bar{\sigma}$.
\end{proof}

\begin{lem} \label{lem:isomfin}
There only finitely many possibilities for the isomorphism class $\varrho_{\mathbb{R}} \circ x_{\recht{cent}}$ of $\mathbb{S}$ as we let $x$ range over all possible elements over all possible $X^+$, for $(X^+,\varrho) \in \bar{\Omega}$.
\end{lem}

\begin{proof}
Let $(Y ^+,\sigma)$ be an element of $\bar{\Omega}_{\recht{der}}$. 
The isomorphism class of the symplectic representation $\sigma_{\mathbb{R}} \circ y$ of $\mathbb{S}$ does not depend on the choice of $y \in Y^+$.
Since $\bar{\Omega}_{\recht{der}}$ is finite, in this way we obtain only finitely many symplectic representations of $\mathbb{S}$.
If we apply this to an $(X^+,\varrho) \in \bar{\Omega}$ and a $x \in X^+$, then $\varrho_{\mathbb{R}} \circ x_{\recht{der}}$ is one of these representations.
It follows that there are only finitely many options for the isomorphism class of $\varrho_{\mathbb{R}} \circ x_{\recht{der}}$ as we let $(X^+,\varrho)$  range over all possible options (it again does not depend on the choice of $x \in X^+$).
On the other hand, there is only one possibility for the isomorphism class $\varrho_{\mathbb{R}} \circ x$; it is the representation of $\mathbb{S}$ corresponding to a polarised Hodge structure with weights $\{(1,0),(0,1)\}$ and dimension $2g$.
We know that the representations $\varrho_{\mathbb{R}} \circ x$, $\varrho_{\mathbb{R}} \circ x_{\recht{der}}$ and $\varrho_{\mathbb{R}} \circ x_{\recht{cent}}$ are given by their multisets of weights $\Sigma, \Sigma_{\recht{der}}, \Sigma_{\recht{cent}} \subset \recht{X}_*(\mathbb{S})$. 
The fact that $x \circ n$ factors as $x_{\recht{cent}} \times x_{\recht{der}}$ implies the following property of these multisets:
There exist orderings $\Sigma_{\recht{der}} = \{\sigma_{1,\recht{der}},\cdots,\sigma_{2g,\recht{der}}\}$ and $\Sigma_{\recht{cent}} = \{\sigma_{1,\recht{cent}},\cdots,\sigma_{2g,\recht{cent}}\}$ such that $n\Sigma = \{\sigma_{i,\recht{der}}+\sigma_{i,\recht{cent}}: i \leq 2g \}$.
By the above, there are only finitely many possibilities for both $\Sigma$ and $\Sigma_{\recht{der}}$, hence there are only finitely many possibilities for $\Sigma_{\recht{cent}}$.
\end{proof}

\begin{lem} \label{lem:bothfin}
The set $\recht{Aut}(G)\backslash(\bar{\Omega}_{\recht{der}} \times \bar{\Omega}_{\recht{cent}})$ is finite.
\end{lem}

\begin{proof}
Let $(X^+,\varrho) \in \bar{\Omega}$, and let $(x_{\recht{cent}} \circ \mu, \varrho_{Z})$ be its associated element of $\bar{\Omega}_{\recht{cent}}$. 
By Lemma \ref{lem:isomfin} we know that there are only finitely many options for the symplectic representation $\nu:= \varrho_{\mathbb{C}} \circ x_{\recht{cent},\mathbb{C}} \circ \mu \colon\mathbb{G}_{\recht{m},\mathbb{C}} \rightarrow \mathtt{GSp}_{2g,\mathbb{C}}$.
Let $\bar{\Omega}_{\recht{cent},\nu} \subset \bar{\Omega}_{\recht{cent}}$ be the subset of all $(x_{\recht{cent},\mathbb{C}} \circ \mu,\varrho|_Z)$ that yield this $\nu$. Each of these is closed under the action of $\recht{Aut}(Z)$.
Since each $\nu$ is defined over $\mathbb{Q}$, Lemma \ref{lem:torfin} tells us that each $\recht{Aut}(Z) \backslash \bar{\Omega}_{\recht{cent},\nu}$ is finite.
Since the image of $\recht{Aut}(G)$ has finite index in $\recht{Aut}(Z)$ by Lemma \ref{lem:autgfin}, we know that $\recht{Aut}(G) \backslash \bar{\Omega}_{\recht{cent}}$ is finite. 
Furthermore, we know from Lemma \ref{lem:derfin} that $\bar{\Omega}_{\recht{der}}$ is finite, so $\recht{Aut}(G)\backslash(\bar{\Omega}_{\recht{cent}} \times \bar{\Omega}_{\recht{der}})$ is finite.
\end{proof}

\begin{proof}[Proof of Proposition \ref{prop:main2rat}]
Since $\varphi$ is finite by Lemma \ref{lem:phifin}, in particular the induced map $\recht{Aut}(G)\backslash\bar{\Omega} \rightarrow \recht{Aut}(G)\backslash(\bar{\Omega}_{\recht{der}} \times \bar{\Omega}_{\recht{cent}})$. Since the codomain of this map is finite by Lemma \ref{lem:bothfin}, we find that its domain is finite as well.
\end{proof}

\subsection{The integral case} \label{ss:integralcase}

In this subsection we prove Theorem \ref{thm:main2}.
Recall that as a complex analytic space we can view $\mathcal{A}_{g,n}$ as a disjoint union of spaces of the form $\Gamma\backslash\mathcal{H}_g^+$, where $\Gamma$ is a congruence subgroup of $\recht{GSp}_{2g}(\mathbb{Z})$.
As before, for a connected Shimura triple $(G,X^+,\varrho)$, let $Y_{\Gamma}(G,X^+,\varrho)$ be the image of $\varrho(X^+)$ in $\Gamma\backslash \mathcal{H}^+_g$.
Let $(G,X^+,\varrho)$ and $(G',\varrho',X'^+,\varrho')$ be two special triples, then $Y_{\Gamma}(G,X^+,\varrho) = Y_{\Gamma}(G',X'^+,\varrho')$ if and only if $G \cong G'$, and if under such an identification we have $(X^+,\varrho) = (X'^+,\varrho')$ in $\recht{Aut}(G)\backslash\Omega(G,g)/\Gamma$. We get a natural map
\begin{equation}
\mathtt{GMT}\colon \recht{Aut}(G)\backslash \Omega(G,g)/\Gamma \rightarrow \recht{Mdl}(G)
\end{equation}
by sending $(X^+,\varrho)$ to the generic (integral) Mumford--Tate group of $Y_{\Gamma}(G,X^+,\varrho)$; this is the same as the generic Mumford--Tate group of $X^+$.
We may also describe the map in the terminology of section \ref{s:background} as follows.
Let $(X^+,\varrho) \in \Omega(G,g)$, and consider the standard representation $V := \mathbb{Q}^{2g}$ of $\mathtt{GSp}_{2g,\mathbb{Q}}$ with its lattice $\Lambda := \mathbb{Z}^{2g}$; then $\mathtt{GMT}(X^+,\varrho) = \mathtt{mdl}_{\varrho(G)}(\Lambda)$.
We can also understand Hecke correspondences in this way:
recall that special subvarieties equivalent to $Y_{\Gamma}(G,X^+,\varrho)$ under Hecke correspondence are of the form $Y_{\Gamma}(G,X^+,\recht{inn}(a)^{-1}\circ \varrho)$ for some $a \in \recht{GSp}_{2g}(\mathbb{Q})$. For such a connected Shimura triple we get
\begin{equation} \label{eq:gmtmdl}
\mathtt{GMT}(X^+,\recht{inn}(a)^{-1} \circ \varrho) = \mathtt{mdl}_{a^{-1}\varrho(G)a}(\Lambda) = \mathtt{mdl}_{\varrho(G)}(a\Lambda).
\end{equation}
Furthermore, the map $\mathtt{GMT}$ allows us to consider Theorem \ref{thm:main2} as a consequence of the following result.

\begin{thm} \label{thm:main2mod}
Let $G$ be a connected reductive group over $\mathbb{Q}$, and let $\Gamma \subset \recht{GSp}_{2g}(\mathbb{Z})$ be a congruence subgroup.
Then the map $\mathtt{GMT}\colon \recht{Aut}(G)\backslash \Omega(G,g)/\Gamma \rightarrow \recht{Mdl}(G)$ is finite.
\end{thm}

\begin{proof}[Proof of Theorem \ref{thm:main2} from Theorem \ref{thm:main2mod}] The Shimura variety $\mathcal{A}_{g,n}$ is a finite disjoint union of connected Shimura varieties $\Gamma\backslash\mathcal{H}_g^+$. 
We need to show that for every $\Gamma$ and for every group scheme $\mathscr{G}$ over $\mathbb{Z}$ there are only finitely many special subvarieties of $\Gamma\backslash\mathcal{H}_g^+$ whose generic Mumford--Tate group is isomorphic to $\mathscr{G}$.
Let $G$ be the generic fibre of $\mathscr{G}$; then every such special subvariety is of the form $Y_{\Gamma}(G,X^+,\varrho)$ for some $(X^+,\varrho) \in \Omega(G,g)$. The Theorem now follows from Theorem \ref{thm:main2mod}.
\end{proof}

Let $\Gamma$ be a congruence subgroup of $\recht{GSp}_{2g}(\mathbb{Z})$.
Write $M(\Gamma) := \Gamma\backslash\mathcal{H}_g^+$; this is a real analytic space.
If $\Gamma$ is small enough, then $M(\Gamma)$ is a connected Shimura variety.
Let $\hat{\mathcal{H}}^+_g$ be the subspace $\recht{GL}_{2g}(\mathbb{R})^+ \cdot \mathcal{H}^+_g$ of $\recht{Hom}(\mathbb{S},\mathtt{GL}_{2g,\mathbb{R}})$, let $\hat{\Gamma}$ be a congruence subgroup of $\recht{GL}_{2g}(\mathbb{Z})$, and define $\hat{M}(\hat{\Gamma}) := \hat{\Gamma}\backslash\hat{\mathcal{H}}^+_g$. 
This is a real analytic space, but for $g > 1$ it will not have the structure of a connected Shimura variety.

\begin{lem} \label{lem:anfin}
Let $\Gamma \subset \recht{GSp}_{2g}(\mathbb{Z})$ be a congruence subgroup, and let $\hat{\Gamma} \subset \recht{GL}_{2g}(\mathbb{Z})$ be a congruence subgroup containing $\Gamma$. Then the map of real analytic spaces $M(\Gamma) \rightarrow \hat{M}(\hat{\Gamma})$ is finite.
\end{lem}

\begin{proof}
It suffices to prove this for $\Gamma = \recht{GSp}_{2g}(\mathbb{Z})$ and $\hat{\Gamma} = \recht{GL}_{2g}(\mathbb{Z})$. 
For these choices of congruence subgroups we have (see \cite{edixhovenyafaev2003}):
\begin{align*}
M(\Gamma) &\cong \left\{\textrm{princ.pol. Hodge structures of type $\{(0,1),(1,0)\}$ on $\mathbb{Z}^{2g}$}\right\}/{\cong},\\
\hat{M}(\hat{\Gamma}) &\cong \left\{\textrm{Hodge structures of type $\{(0,1),(1,0)\}$ on $\mathbb{Z}^{2g}$}\right\}/{\cong}.
\end{align*}
In this terminology, the natural map $M(\Gamma) \rightarrow \hat{M}(\hat{\Gamma})$ is just forgetting the polarisation.
By \cite[Thm.~18.1]{milne1986}, a polarisable $\mathbb{Z}$-Hodge structure of type $\{(0,1),(1,0)\}$ has only finitely many principal polarisations (up to automorphisms of polarised Hodge structures), which proves the Lemma.
\end{proof}

\begin{proof}[Proof of Theorem \ref{thm:main2mod}]
By Proposition \ref{prop:main2rat} it suffices to show that for every $\recht{GSp}_{2g}(\mathbb{Q})$-orbit $B$ in $\recht{Aut}(G) \backslash \Omega$ the map $\mathtt{GMT}\colon B/\Gamma \rightarrow \recht{Mdl}(G)$ is finite.
Let $(X^+,\varrho)$ be an element of such a $B$, and let $N$ be the scheme-theoretic normaliser of $\varrho(G)$ in $\mathtt{GSp}_{2g,\mathbb{Q}}$.
Then as a right $\recht{GSp}_{2g}(\mathbb{Q})$-set we can identify $B$ with $N(\mathbb{Q})\backslash\recht{GSp}_{2g}(\mathbb{Q})$, and under this identification we find
\begin{align}
B/\Gamma &\DistTo N(\mathbb{Q})\backslash\recht{GSp}_{2g}(\mathbb{Q})/\Gamma \label{eq:bmodgamma}\\
(X^+,\varrho) \cdot a \cdot \Gamma &\mapsto N(\mathbb{Q})a\Gamma. \nonumber
\end{align}
Let $V$ and $\Lambda$ be as in the beginning of this section,  and let $\hat{N}$ be the scheme-theoretic normaliser of $\varrho(G)$ in $\mathtt{GL}_{2g,\mathbb{Q}}$. By considering every lattice as being of the form $g \cdot \Lambda$ for some $g \in \recht{GL}_{2g}(\mathbb{Z})$, we get an identification
\begin{equation*}
\recht{Lat}_{\hat{N}}(V) \cong \hat{N}(\mathbb{Q})\backslash\recht{GL}_{2g}(\mathbb{Q})/\recht{GL}_{2g}(\mathbb{Z}).
\end{equation*}
From (\ref{eq:gmtmdl}) we see that the map $\mathtt{GMT}\colon B/\Gamma \rightarrow \recht{Mdl}(G)$ sends a double coset $N(\mathbb{Q})a\Gamma$, considered as an element of $B/\Gamma$ via (\ref{eq:bmodgamma}), to $\mathtt{mdl}_G(a\Lambda)$. As such we can decompose $\mathtt{GMT}$ into
\begin{align}
B/\Gamma &\DistTo  N(\mathbb{Q})\backslash\recht{GSp}_{2g}(\mathbb{Q})/\Gamma \nonumber \\
&\twoheadrightarrow N(\mathbb{Q})\backslash\recht{GSp}_{2g}(\mathbb{Q})/\recht{GSp}_{2g}(\mathbb{Z}) \label{eq:chain1}\\
&\rightarrow \hat{N}(\mathbb{Q})\backslash\recht{GL}_{2g}(\mathbb{Q})/\recht{GL}_{2g}(\mathbb{Z}) \label{eq:chain2}\\
&\DistTo \recht{Lat}_{\hat{N}}(V) \nonumber \\
&\stackrel{\mathtt{mdl}_G}{\rightarrow} \recht{Mdl}(G). \label{eq:chain3}
\end{align}
The map in (\ref{eq:chain1}) is finite because $\Gamma$ is of finite index in $\recht{GSp}_{2g}(\mathbb{Z})$, and the map in (\ref{eq:chain3}) is finite by Theorem \ref{thm:main1}. Hence it suffices to show that the map in (\ref{eq:chain2}) is finite; denote this map by $f$. Let $\mathcal{Z}$ be the set of connected real analytic subspaces of $M(\recht{GSp}_{2g}(\mathbb{Z}))$, and let $\hat{Z}$ be the set of connected real analytic subspaces of $\hat{M}(\recht{GL}_{2g}(\mathbb{Z}))$. By Lemma \ref{lem:anfin} the induced map $z\colon \mathcal{Z} \rightarrow \hat{\mathcal{Z}}$ is finite as well. There are injective maps
\begin{align}
\iota\colon N(\mathbb{Q})\backslash\recht{GSp}_{2g}(\mathbb{Q})/\recht{GSp}_{2g}(\mathbb{Z}) &\hookrightarrow \mathcal{Z},\\
\hat{\iota}\colon \hat{N}(\mathbb{Q})\backslash\recht{GL}_{2g}(\mathbb{Q})/\recht{GL}_{2g}(\mathbb{Z}) &\hookrightarrow \hat{\mathcal{Z}}
\end{align}
where $\iota$ sends an $a \in \recht{GSp}_{2g}(\mathbb{Q})$ to the image of $a^{-1}\varrho(X^+)a$ in $\recht{GSp}_{2g}\backslash\mathcal{H}^+_g$, and $\hat{\iota}$ is defined analogously. Then $z \circ \iota  = \hat{\iota} \circ f$, and since $z$ is finite and $\iota$, $\hat{\iota}$ are injective, we see that $f$ is finite, which proves the Theorem.
\end{proof}
\DeclareRobustCommand{\VAN}[3]{#3}

\begin{rem}
Let $\mathbb{L}$ be the set of prime numbers. By applying Remark \ref{rem:locmod} rather than Theorem \ref{thm:main1}, we can also prove that for a collection of $\mathbb{Z}_{\ell}$-group schemes $(\msc{G}_{\ell})_{\ell \in \mathbb{L}}$, there are at most finitely many special subvarieties $Y \subset \mathcal{A}_{g,n}$ with $\mathtt{GMT}(Y)_{\mathbb{Z}_{\ell}} \cong \msc{G}_{\ell}$. In particular, if $A$ is a $g$-dimensional abelian variety over a number field with a degree $n$ polarisation, and $\mathtt{G}_{\ell}(A)$ is its $\ell$-adic Galois monodromy group (as a group scheme over $\mathbb{Z}_{\ell}$), then there are only finitely many $Y$ with $\mathtt{GMT}(Y)_{\mathbb{Z}_{\ell}} \cong \mathtt{G}_{\ell}(A)$ for all $\ell \in \mathbb{L}$. On the other hand, the Mumford--Tate conjecture states that if $\mathtt{MT}(A)$ is the (integral) Mumford--Tate group of $A$, then $\mathtt{MT}(A)_{\ell} \cong \mathtt{G}_{\ell}(A)$ for all $\ell$ (see \cite{cadoretmoonen2015} for the integral version of the conjecture). In particular, it implies that at least one $Y$ as above exists, namely the special closure of the point on $\mathcal{A}_{g,n}$ corresponding to $A$.
\end{rem}


\begin{thebibliography}{00}
\bibitem{borel1963} Armand Borel. Some finiteness properties of adele groups over number fields. \emph{Publications Mathématiques de l'IHÉS}, 16:5--30, 1963.
\bibitem{bourbaki2006} Nicholas Bourbaki. \emph{Groupes et algèbres de Lie: Chapitres 7 et 8}. Eléments de Mathématique. Springer: Berlin, Heidelberg, 2006.
\bibitem{brown1989} Kenneth S. Brown. \emph{Buildings}. Springer: New York City, 1989.
\bibitem{bruhattits1972} François Bruhat and Jacques Tits. Groups réductifs sur un corps local. \emph{Publications Mathématiques de l'IHÉS}, 41(1):5--251, 1972.
\bibitem{cadoretmoonen2015} Anna Cadoret and Ben Moonen. Integral and adelic aspects of the Mumford--Tate conjecture. \emph{Journal of the Institute of Mathematics of Jussieu}, 1--22, 2018.
\bibitem{conrad2011} Brian Conrad. \emph{Non-split reductive groups over $\mathbb{Z}$}, 2011. Available at \url{http://math.stanford.edu/~conrad/papers/redgpZsmf.pdf}.
\bibitem{deligne} Pierre Deligne. Variétés de Shimura: interprétation modulaire, et techniques de construction de modèles canoniques. In \emph{Automorphic forms, representations and $L$-functions: Symposium in Pure Mathematis held at Oregon State University, July 11--August 5, 1977, Corvallis, Oregon}, volume 33, part 1 of \emph{Proceedings of Symposia in Pure Mathematics}, pp.~247--289. American Mathematical Society: Providence, 1979.
\bibitem{edixhovenyafaev2003} Bas Edixhoven and Andrei Yafaev. Subvarieties of Shimura varieties. \emph{Annals of Mathematics}, 157(2):621--645, 2003.
\bibitem{etingof2011} Pavel Etingof, Oleg Golberg, Sebastian Hensel, Tiankai Liu, Alex Schwendner, Dmitry Vaintrob, and Elena Yudovina. \emph{Introduction to representation theory}. Volume 59 of \emph{Student Mathematical Library}. American Mathematical Society: Providence, 2011.
\bibitem{fomina1997} Tat'yana Vladimirovna Fomina. Integral forms of linear algebraic groups. \emph{Mathematical Notes}, 61(3):346--351, 1997. Translated by A.I. Shtern from the Russian original: Целые формы линейных алгебраических групп. \emph{Математические заметки}, 61(3):424--430, 1997.
\bibitem{gross1996} Benedict H. Gross. Groups over $\mathbb{Z}$. \emph{Inventiones mathematicae}, 124:263--279, 1996.
\bibitem{humphreys1972} James E. Humphreys. \emph{Introduction to Lie algebras and representation theory}. Volume 9 of \emph{Graduate Texts in Mathematics}. Springer: New York City, 1972.
\bibitem{humphreys1975} James E. Humphreys. \emph{Linear algebraic groups}. Volume 21 of \emph{Graduate Texts in Mathematics}. Springer: New York City, 1975.
\bibitem{knop2006} Friedrich Knop. Classification of multiplicity free symplectic representations. \emph{Journal of Algebra}, 301(2):531--553, 2006.
\bibitem{milne1986} James S. Milne. Abelian varieties. In \emph{Arithmetic geometry}, pp.~103--150. Springer: New York City, 1986.
\bibitem{milne2004introduction} James S. Milne. \emph{Introduction to Shimura varieties}. 2004. Available at \url{www.jmilne.org/math/}.
\bibitem{milne2013} James S. Milne. \emph{Lie algebras, algebraic groups, and Lie groups}. 2013. Available at \url{www.jmilne.org/math/}.
\bibitem{milne2016} James S. Milne. \emph{Etale cohomology}. Volume 33 of \emph{Princeton Mathematical Series}. Princeton University Press: Princeton, 2016.
\bibitem{milne2017} James S. Milne. \emph{Algebraic groups}. Volume 170 of Cambridge Studies in Advanced Mathematics. Cambridge University Press: Cambridge, 2017.
\bibitem{moonen2004introduction} Ben Moonen. \emph{An introduction to Mumford--Tate groups}, 2004. Available at \url{https://www.math.ru.nl/~bmoonen}.
\bibitem{platonovrapinchuk1994} Vladimir Petrovich Platonov and Andrei Stepanovic Rapinchuk. \emph{Algebraic groups and number theory}. Volume 139 of \emph{Pure and Applied Mathematics}. Academic Press: San Diego, 1994. Translated by Rachel Rowen from the Russian original: \emph{Алгебраические группы и теория чисел}. Nauka: Moscow, 1991.
\bibitem{rousseau1977} Guy Rousseau. \emph{Immeubles des groupes réductifs sur les corps locaux}. Université Paris XI, UER Mathémathique, 1977.
\bibitem{serre1997} Jean-Pierre Serre. \emph{Galois cohomology}. Springer Monographs in Mathematics. Springer: Berlin, Heidelberg, 1997.
\bibitem{sga3} Michel Demazure and Michael Artin. \emph{Schémas en groupes (SGA3)}. Springer: Berlin, Heidelberg, New York, 1970.
\bibitem{springer1998} Tonny Albert Springer. \emph{Linear algebraic groups} (2nd edition). Modern Birkhäuser Classics. Birkhäuser: Basel, 1998.
\bibitem{stacks-project} The Stacks Project Authors. \emph{Stacks project}. Available at \url{http://stacks.math.columbia.edu}, 2019.
\bibitem{tits1979} Jacques Tits. Reductive groups over local fields. In \emph{Automorphic forms, representations and $L$-functions: Symposium in Pure Mathematics held at Oregon State University, July 11--August 5, 1977, Corvallis, Oregon}, volume 33, part 1 of \emph{Proceedings of Symposia in Pure Mathematics}, pages 247--289. American Mathematical Society: Providence, 1979.
\bibitem{user27056} user27056. Is the normalizer of a reductive subgroup reductive? 2012. Available at \url{https://mathoverflow.net/questions/114243/is-the-normalizer-of-a-reductive-subgroup-reductive}.
\end{thebibliography}
\end{document}